\documentclass[11pt]{amsart}

\usepackage{amsmath, amsthm, amscd, amssymb, amsfonts, amsxtra, amssymb, latexsym}
\usepackage{enumerate}
\usepackage{verbatim}
\usepackage{textcomp}
\usepackage{graphicx,epsfig,tikz}
\usepackage{tgpagella}
\usepackage{pstricks,pst-node,pst-text,pst-3d,pst-plot}
\usepackage{fancybox,shadow,color,graphicx,psfrag,float}
\usepackage{inputenc, t1enc}

\newcommand{\na}{\mathbb{N}}
\newcommand{\re}{\mathbb{R}}

\newcommand{\sg}{\mathbb{S}}

\newcommand{\rr}{\mathcal{L}}
\newcommand{\ff}{\mathbb{F}}

\newcommand{\cod}{\mathcal{C}}
\newcommand{\F}{\mathcal{F}}

\newcommand{\sop}{\operatorname{Sup}}

\newcommand{\aut}{\operatorname{Aut}}
\newcommand{\gal}{\operatorname{Gal}}

\newcommand{\sk}{\smallskip}
\newcommand{\msk}{\medskip}

\newtheorem{thm}{Theorem}[section]
\newtheorem{prop}[thm]{Proposition}
\newtheorem{lem}[thm]{Lemma}
\newtheorem{coro}[thm]{Corollary}
\theoremstyle{definition}
\newtheorem{rem}[thm]{Remark}
\newtheorem{exam}[thm]{Example}
\newtheorem{defi}[thm]{Definition}

\theoremstyle{remark}

\newcommand{\legendre}[2]{\left( \tfrac{#1}{#2} \right)}

\usepackage{color}

\begin{document}
\numberwithin{equation}{section}
\title[Block-transitive AG-codes attaining the TVZ-bound]{Block-transitive algebraic geometry codes attaining the Tsfasman-Vladut-Zink bound}
\author[M.\@ Chara, R.\@ Podest\'a, R.\@ Toledano]{Mar\'ia Chara, Ricardo Podest\'a, Ricardo Toledano}
\dedicatory{\today}

\thanks{Partially supported by CONICET, FONCyT, UNL CAI+D 2014, SECyT-UNC, CSIC}
\address{Mar\'ia Chara -- Departamento de Matem\'atica, FIQ-UNL-CONICET, FICH (3000) Santa Fe, Argentina. \\
{\it E-mail: mchara@santafe-conicet.gov.ar}}
\address{Ricardo Podest\'a -- FaMAF -- CIEM (CONICET), Universidad Nacional de C\'ordoba, (5000) C\'ordoba, Argentina. 
	{\it E-mail: podesta@famaf.unc.edu.ar}}
\address{Ricardo Toledano -- Departamento de Matem\'atica, Facultad de Ingenier\'ia Qu\'imica, UNL, (3000) Santa Fe, Argentina.  {\it E-mail: ridatole@gmail.com}}

\begin{abstract}
We study the asymptotic behaviour of a family of algebraic geometry codes, which we call block-transitive, that generalizes the classes of transitive and quasi-transitive codes.
We prove, by using towers of algebraic function fields, that there are sequences of codes in this family attaining the Tsfasman-Vladut-Zink bound over finite fields of square cardinality. We give the exact length of these codes as well as explicit lower bounds for their parameters.  
\end{abstract}

\maketitle

\section{Introduction}
Let $\ff_q$ denote a finite field with $q$ elements. 
In 1982, Tsfasman, Vladut and Zink proved, in their remarkable paper \cite{TVZ}, that there are sequences of linear codes over $\ff_{p^2}$, $p$ a prime number, whose transmission rates satisfy what is known today as the Tsfasman-Vladut-Zink bound. For finite fields of square cardinality bigger than 49, this bound is above the well known Gilbert-Varshamov bound  
in some interval and thus the work \cite{TVZ} showed, for the first time, the existence of a sequence of linear codes beating the Gilbert-Varshamov bound for infinitely many finite fields. This important result was achieved by considering algebraic geometry codes associated to  towers of modular curves over 
$\ff_{p^2}$. However, as of today, 
little is known about the structure of these codes besides linearity and the more we know about structural properties of a family of codes the better because they can be very helpful for the design of efficient coding and decoding methods. 

Almost a quarter of a century later, a further step in this direction was given by Stichtenoth in \cite{Sti06}. There, he reproved the result of Tsfasman, Vladut and Zink by showing that there are sequences of transitive codes and self-dual codes (among others) over finite fields $\ff_q$ of square cardinality $q>4$ attaining the Tsfasman-Vladut-Zink bound. Since he used the Galois closure of a recursive tower of function fields over $\ff_q$, there is an improvement with respect to the codes given in \cite{TVZ}. Namely, we have, on the one side, a sequence of linear codes with a richer structure and, on the other side, a sequence of function fields which is more manageable, in principle, than the sequence of modular towers used in 
\cite{TVZ}.   

In this paper, we give another step by considering a different family of linear codes and showing that there are sequences of codes in this family attaining the Tsfasman-Vladut-Zink bound over finite fields of square cardinality $\ff_{q^2}$ with $q$ a prime power. These codes are finite direct sums of transitive codes and we call them \textit{block-transitive} (see Definition \ref{block codes}). 
The advantage we have with this family of codes, with respect to others in previous works, lies in that we use the Galois closure of optimal towers, tamely and wildly ramified, defined in \cite{GS2} and \cite{GS96} which have been 
studied in detail in \cite{BB} and \cite{BB11}. Thus we know explicitly, for instance, the degree of every extension  as well as the structure of the Galois groups associated to these extensions, which are subgroups of the group of automorphisms of these codes. 
This is a valuable feature because we now have at our disposal more tools for a better and deeper understanding of the structure of these codes. 
At this point it is worth to mention that this is not available, to the best of our knowledge, for the transitive codes given in \cite{Sti06} because 
the Galois closure of the function fields of an optimal tower with respect to a proper subfield of the original field of rational functions was considered, and this is not the case treated in \cite{BB11}. Another point worth mentioning is that the use of block-transitive codes allow us to use simpler arguments in more general results.

\sk

An outline of the paper is as follows. Sections 2--3 are preparatory while the main results are in Sections 4--7.
More precisely, in Section 2 we briefly recall the basic definitions on algebraic codes and towers of algebraic function fields 
and their corresponding asymptotic behaviours. Also, in Definition \ref{defi} we introduce the notion of a $(\lambda,\delta)$-bound, 
with the famous Tsfasman-Vladut-Zink bound being a particular case. In the next section we review algebraic-geometry codes (AG-codes for short) and the standard way to get asymptotically good sequences of AG-codes from sequences of towers, attaining such a $(\lambda,\delta)$-bound. We also present some concrete realizations of good AG-codes using either recursive towers (Example 3.5) or lower bounds for the Ihara function (Examples 3.6 and 3.7).

Section 4 begins with the definition of block-transitive codes, that we use throughout the paper, together with some examples. 
Then, in Theorem \ref{generalgoodag} we give sufficient conditions on the splitting and ramification locus of a tower of function fields 
to be suitable for the construction of good $r$-block transitive codes over a given finite field. In the next 2 sections we apply this result to prove the existence of families of block-transitive codes that are asymptotically good. In fact, using Theorem \ref{generalgoodag} and two quadratic optimal towers given in \cite{GS96} and \cite{GS2} we prove, in Theorem \ref{goodag-psquare}, the existence of sequences of $r$-block transitive codes with $r=q^2-q$ 
(resp.\@ $r=2(p-1)$, $p$ an odd prime) attaining the 
Tsfasman-Vladut-Zink bound over $\ff_{q^2}$ (resp.\@ $\ff_{p^2}$). 
Then, in Theorem \ref{goodAGanychar} we use good tamely ramified Hilbert class field towers with non zero splitting rate, given in \cite{AM}, to prove the existence of good $r$-block transitive codes over certain finite fields, not necessarily of square cardinality, of both odd and even characteristic and for $r=2n$ or $r=3n$, respectively. The price we must pay for considering a wider class of finite fields is reflected in that we get weaker lower bounds for their performance compared to the ones obtained in Theorem \ref{goodag-psquare} and \cite{Sti06} for the case of transitive codes.  
Finally we make some remarks on the existence of sequences of algebraic geometry codes with good performance over prime fields.

\section{Asymptotic behaviour of codes and towers} 
A linear code of \textit{length} $n$, \textit{dimension} $k$ and \textit{minimum distance} $d$ over $\ff_q$ 
is an $\ff_q$-linear subspace $\cod $ of $\ff_q^n$ with $k=\dim \cod $ and $d=\min \{d(c,c'):c,c'\in \cod, c\ne c'\}$, where $d$ is the Hamming distance in $\ff_q^n$. The elements of $\cod$ are usually called codewords 
and it is customary to say that $\cod$ is an $[n,k,d]$-\textit{code} over $\ff_q$.

It is convenient to normalize the dimension and minimum distance of a code $\cod$ with respect to the length, so that 
the {\em information rate} $R=R(\cod)\colon\!\!\!=k/n$ and the {\em relative minimum distance} $\delta=\delta(\cod)\colon\!\!\!=d/n$ 
of $\cod$ are in the unit interval for any $[n,k,d]$-code $\cod$. The goodness of an $[n,k,d]$-code over $\ff_q$ is usually measured according to how big the sum $R+\delta$ is, where
$$0< R+\delta= \tfrac kn + \tfrac dn \le 1+\tfrac 1n.$$ 
Since $k$ and $d$ can not be arbitrarily large for fixed $n$ it is natural to allow arbitrarily large lengths. One way to do this is by considering a  sequence $\{\cod_i\}_{i=0}^{\infty}$ of $[n_i,k_i,d_i]$-codes over $\ff_{q}$ such that $n_i\rightarrow\infty$ as $i\rightarrow\infty$ and see how big the sum $k_i/n_i+d_i/n_i$ can be as $n_i\rightarrow\infty$.

\subsubsection*{On $(\ell,\delta)$-bounds for codes} 
A sequence $\{\cod_i\}_{i=0}^{\infty}$ of $[n_i,k_i,d_i]$-codes over $\ff_{q}$ is 
{\em asymptotically good over $\ff_q$} if $$ \liminf_{i\rightarrow\infty}\frac{k_i}{n_i}>0 \qquad \text{and} 
\qquad  \liminf_{i\rightarrow\infty}\frac{d_i}{n_i}>0  $$
where $n_i\rightarrow\infty$ as $i\rightarrow\infty$. In view of this, we propose the following

\begin{defi}\label{defi}
Let $\ell,\delta\in (0,1)$ such that $\ell>\delta$. A sequence $\{\cod_i\}_{i=0}^{\infty}$ of 
$[n_i,k_i,d_i]$-codes over $\ff_{q}$ is said to attain an 
{\em $(\ell,\delta)$-bound over $\ff_q$}  if 
$$ \liminf_{i\rightarrow\infty}\frac{k_i}{n_i}\geq \ell-\delta \qquad \text{and} \qquad  \liminf_{i\rightarrow\infty}\frac{d_i}{n_i}\geq \delta \,. $$ 
In particular, $\{\cod_i\}_{i=0}^{\infty}$ is asymptotically good.
\end{defi}

Thus, a lower bound for the above mentioned goodness of a code is obtained  for a sufficiently large  index $i$ since, as $i\rightarrow\infty$ we have  
$$ \frac{k_i}{n_i}+\frac{d_i}{n_i}\sim \ell-\delta+\delta=\ell.$$ 
In other words, an $(\ell,\delta)$-bound says more than the good asymptotic behaviour of the sequence of codes 
$\{\cod_i\}_{i=0}^{\infty}$. 
In fact, any $(\ell, \delta)$-bound gives a linear function as a lower bound for the \textit{Manin's function} $\alpha_q(\delta)$; 
namely, for $0\leq\delta\leq 1$, we have 
\begin{equation} \label{alfaq}
\alpha_q(\delta)\geq \ell-\delta.
\end{equation}

Consider now the so called \textit{Ihara's function} (see, for instance, \cite{NX} and \cite{Sti})
\begin{equation} \label{Aq}
A(q) \colon\!\!= \limsup_{g\rightarrow\infty} \frac{N_q(g)}{g} \,, 
\end{equation}
where $N_q(g)$ is the maximum number of rational places that a function field over $\ff_q$ of genus $g$ can have. 
It is known that 
\begin{equation} \label{ihara bounds}
c\log q \le A(q) \le \sqrt{q}-1
\end{equation}
for some positive constant $c$, where the first inequality is due to Serre (\cite{Serre85}) and the second one to Drinfeld and Vladut 
(\cite{DV83}).
Moreover, the equality holds in the upper bound for fields of square cardinality, i.e. 
\begin{equation} \label{Aq2} 
A(q^2) = q-1.
\end{equation}
There are some improvements and refinements of lower bounds for $A(q)$, as can be seen in \cite{NX}.  

The famous {\em Tsfasman-Vladut-Zink bound over} $\ff_q$ (\textit{TVZ-bound} for short) is an example of an $(\ell,\delta)$-bound with $\delta > 0$ and 
\begin{equation} \label{lAq} 
\ell = 1 - A(q)^{-1}
\end{equation}  
if $A(q)>1$. 
For squares $q\geq 49$, the TVZ-bound over $\ff_q$ improves the Gilbert-Varshamov bound which was considered, for some time, the best lower bound for Manin's function (see \cite{Sti}, \cite{Sti06} and \cite{TVN}). The GV-bound can also be seen as an $(\ell,\delta)$-bound with 
$$\ell=\delta+1-H_q(\delta),$$ 
where $H_q$ is the $q$-ary entropy function
$$H_q(x) = x\log_q(q-1)-x\log_q(x) -(1-x)\log_q(1-x)$$ 
for $0 \le x \le 1-\tfrac{1}{q}$ and $H_q(0)=0$.

\subsubsection*{Towers} 
A function field (of one variable) over $K$ is a finite algebraic extension $F$ of the rational function field $K(x)$, where $x$ is a transcendental element over $K$. 
Following \cite{Sti}, a sequence $\F=\{F_i\}_{i=0}^{\infty}$ of function fields over a $K$ is called a \textit{tower} if for every 
$i\ge 0$ we have 
\begin{enumerate}
\item $F_i \subsetneq F_{i+1}$; 

\sk \item the extension $F_{i+1}/F_{i}$ is finite and separable;

\sk \item $K$ is perfect and algebraically closed in $F_i$; and 

\sk \item the genus $g(F_i)$ of $F_i$ tends to infinity, for  $i\rightarrow \infty$.
\end{enumerate}

A tower $\F=\{F_i\}_{i=0}^{\infty}$ over $K$ is called {\it recursive} if there exist a sequence of transcendental elements 
$\{x_i\}_{i=0}^{\infty}$ over $K$ and a polynomial $H(X,Y)\in K[X,Y]$ such that $F_0=K(x_0)$ and
$$ F_{i+1} = F_i (x_{i+1}) \quad \text{where} \quad H(x_i,x_{i+1})=0  \quad \text{for all } i\geq 0.$$ 

Now we define the concept of asymptotic behaviour of a tower. 
Let $\F=\{F_i\}_{i=0}^{\infty}$ be a  tower of function fields over  $K$. 
The \textit{genus} of $\F$ over $F_0$ is defined as
\begin{equation} \label{gF} 
\gamma(\F) := \lim_{i\rightarrow \infty}\frac{g(F_i)}{[F_i:F_0]}. 
\end{equation}
where $g(F_i)$ stands for the genus of $F_i$. 
Notice that $0<\gamma(\F)\leq \infty$. When $K$ is finite, we denote by $N(F_i)$ the number of rational places 
(i.e., places of degree one) of $F_i$ and the \textit{splitting rate} of $\F$ over $F_0$ is defined as
\begin{equation} \label{srF} 
\nu(\F):=\lim_{i\rightarrow \infty}\frac{N(F_i)}{[F_i:F_0]}.
\end{equation} 
Notice that $0\leq \nu(\F)<\infty$. 

A tower $\F$ is called \textit{asymptotically good} over a finite field $K$ if 
$$\nu(\F)>0 \qquad \text{ and } \qquad \gamma(\F) < \infty.$$ 
Otherwise is called \textit{asymptotically bad}. Equivalently, $\F$ is asymptotically good over $K$ 
if and only if \textit{the limit} of the tower $\F$ is positive, i.e.\@
\begin{equation} \label{limit}
\lambda(\F):=\lim_{i\rightarrow \infty}\frac{N(F_i)}{g(F_i)}=\frac{\nu(\F)}{\gamma(\F)} >0.
\end{equation}
Notice that if $\F$ is a tower of function fields over $\ff_q$ then
$\lambda(\F)\leq A(q)$. The tower is said to be \textit{optimal} if $\lambda(\F)=A(q)$.

\section{Good AG-codes from sequences of function fields} \label{gcfs}
Many families of $\ff_q$-linear codes have been proved to attain the 
TVZ-bound using Goppa ideas for constructing linear codes from the set of rational points of algebraic curves over $\ff_q$. They 
are called \textit{algebraic geometry} codes (or simply AG-codes) and some standard references for them are the books \cite{Sti}, 
\cite{TVN}, \cite{MMR}, \cite{Mo} and \cite{Ste}. 

We now recall their construction using the terminology of function fields (instead of algebraic curves). 
Let $F$ be an algebraic function field over $\ff_q$, let $G$ and $D=P_1+\cdots +P_n$ be disjoint divisors of $F$, 
where $P_1,\ldots,P_n$ are different rational places of $F$. 
Consider now the Riemann-Roch space 
$\rr(G) = \{u\in F \smallsetminus \{0\} : (u) \ge -G\} \cup \{0\}$
associated to $G$, where $(u)$ denotes the principal divisor of $u\in F$. The algebraic-geometry code
defined by $F$, $D$ and $G$ is
\begin{equation}\label{C(DG)}
\cod = C_{\rr}(D,G) = \big \{ \big( u(P_1),u(P_2),\ldots,u(P_n) \big) \in \ff_q^n : u\in \rr(G) \big\},
\end{equation}
with $u(P_i)$ standing for the residue class of $x$ modulo $P_i$. 

An important feature of these $[n,k,d]$-codes is that lower bounds for 
$k$ and $d$ are available in terms of the genus $g=g(F)$ of $F$ and the degree $\deg G$ of $G$.  
More precisely, $k\geq \deg G+1-g$ if $\deg G<n$ and $d\geq n-\deg G$.
Furthermore, if $2g-2<\deg G<n$ then $k=\deg G+1-g$.

\subsubsection*{Good AG-codes from sequences} 
There is a standard way to construct a sequence of AG-codes over a finite field $\ff_q$ attaining an $(\ell,\delta)$-bound, as we show next. We include the proof for the convenience of the reader. 

\begin{prop} \label{lem asint}
Let $\mathcal{F} = \{F_i\}_{i=0}^{\infty}$ be a sequence of algebraic function fields over $\ff_q$ with $\ff_q$ as their full field of constants and such that for each $i\geq 0$ there are
$n_i$ rational places $P_{i,1},\ldots, P_{i,n_i}$ in $F_i$ with $n_i\in\na$. Let $\lambda\in (0,1)$ and suppose that the following three
conditions hold:

(a) $n_i\rightarrow\infty$ as $i\rightarrow\infty$; 

(b) there exists an index $i_0$ such that $\frac{g(F_i)}{n_i} \le \lambda$ 
for all $i\geq i_0$; and
 
(c) for each $i>0$ there is a divisor $G_i$ of $F_i$ which is disjoint with 
$D_i := P_{i,1} + \cdots + P_{i,n_i}$ such that 
$\deg G_i\leq n_i \cdot \alpha(i)$, 
where $\alpha:\na \rightarrow \re$ with $\alpha(i)\rightarrow 0$ as $i\rightarrow\infty$.

\noindent Then, there exists a sequence $\{r_i\}_{i=N}^{\infty} \subset \na$, for certain $N,$ such that $\mathcal{F}$ induces a sequence 
$\mathcal{G}=\{\cod_i\}_{i=N}^\infty$ of asymptotically good AG-codes 
$\cod_i = C_{\rr}(D_i,r_iG_i)$ attaining an $(\ell,\delta)$-bound with $\ell=1-\lambda$.
In particular, if $\lambda=A(q)^{-1}$ the sequence $\mathcal{G}$ attains the TVZ-bound.
\end{prop}

\begin{proof}
Let $\delta\in (0,1)$ be a fixed real number such that $1-\delta>\lambda$.  Note that $\delta$ depends only on $\lambda$. Since $\alpha(i)\rightarrow 0$ as $i\rightarrow\infty$ there exists an index $i_1$ such that $\alpha(i)<1-\delta$ for $i>i_1$. Clearly, we can find a positive integer $r_i$ such that,
for every $i\geq i_1$, we have
\begin{equation}\label{etaerrei}
1-\delta \geq r_i \frac{\deg G_i}{n_i} > 1 - \delta - \alpha(i).
\end{equation}

Let us consider now the AG-code $\cod_i = C_{\rr}(D_i,r_iG_i)$ for $i\geq i_0$. This is a code of length $n_i$ and we have that its minimum distance $d_i$ satisfies $ d_i \geq n_i - r_i \deg G_i$.
From this and the first inequality in \eqref{etaerrei} we see that the relative minimum distance $\cod_i$ satisfies
$$\frac{d_i}{n_i}\geq 1 - r_i \frac{\deg G_i}{n_i} \geq \delta.$$
On the other hand, for the dimension $k_i$ of $\cod_i$ we have
$$ k_i \geq r_i \deg G_i + 1 - g(F_i) > r_i \deg G_i - g(F_i). $$
So that by (b) and the second inequality in \eqref{etaerrei} the information rate of $\cod_i$ satisfies
$$\frac{k_i}{n_i} > \frac{r_i \deg G_i}{n_i} - \frac{g(F_i)}{n_i} \geq \frac{r_i\deg G_i}{n_i} - \lambda \geq 1 - \delta -\alpha(i) -\lambda $$
for $i\ge N=\max \{i_0,i_1\}$. 
Therefore,
$$ \liminf_{i\rightarrow\infty}\frac{k_i}{n_i}\geq 1-\lambda -\delta
\qquad \text{and} \qquad  
\liminf_{i\rightarrow\infty}\frac{d_i}{n_i}\geq\delta, $$
showing that the sequence of AG-codes $\{\cod_i\}_{i=N}^\infty$ attains an $(\ell,\delta)$-bound with $\ell=1-\lambda$. 
The remaining assertions follows directly from \eqref{lAq}.
\end{proof}

\begin{rem}
It is easy to see that in a concrete situation the crucial properties to have are conditions (a) and (b). 
In fact, once (a) and (b) are satisfied, condition (c) is not a problem unless we want the codes $\cod_i$ to have some additional structure besides linearity. 
More precisely, if we have (a) and (b) and we are just looking for a good sequence of linear codes, it suffices to consider $G_i=P_i$, where $P_i$ is a rational place of $F_i$ which is not in the support of $D_i$. 
However, this choice of $G_i$ may not be adequate when, for instance, the invariance of $G_i$ under some $F_0$-embedding of $F_i$ is needed as it happens in the cases of cyclic, transitive or block-transitive AG-codes (see Section \ref{blocktransitivedef} for details).
\end{rem}

\begin{rem}
By the definition of the Ihara function we have $A(q) \ge \frac{n_i}{g(F_i)} \ge \lambda^{-1}$, where $\lambda$ is the constant in (b) of 
Proposition \ref{lem asint}. By the second inequality in \eqref{ihara bounds}, we see that 
$$\lambda \geq \frac{1}{\sqrt{q}-1}.$$  
This bound tell us that the construction of asymptotically good AG-codes over $\ff_q$ provided by 
Proposition~\ref{lem asint} can not be carried out for $q\leq 4$, since $\lambda\in (0,1)$.
\end{rem} 

\begin{rem}
Proposition \ref{lem asint} holds in the case of an asymptotically good tower $\mathcal{F}=\{F_i\}_{i=0}^{\infty}$ of function fields over $\ff_q$ with limit $\lambda(\F)>1$. In fact, if $N_i$ is the number of rational places of $F_i$ then we can take $n_i=N_i-1$, $G_i=Q_i$ (where $Q_i$ is the rational place of $F_i$ not considered among the $n_i$ chosen rational places) and $\alpha(i)=1/n_i$. Then
$$ \lim_{i\rightarrow\infty}\frac{n_i}{g(F_i)} = \lim_{i\rightarrow\infty} \frac{N_i-1}{g(F_i)}=\lambda(\F)>1, $$
so that there is an index $i_0$ such that (b) of Proposition \ref{lem asint} holds for all $i\geq i_0$ with $\lambda=\lambda(\F)^{-1}$. Clearly, items (a) and (c) also holds. 
Further, notice that if (b) holds using the sequence of function fields of a tower $\mathcal{F} = \{F_i\}_{i=0}^{\infty}$ over $\ff_q$, then
$$ \lim_{i\rightarrow\infty}\frac{N_i}{g(F_i)} \geq \liminf_{i\rightarrow\infty}\frac{n_i}{g(F_i)} \ge \lambda^{-1}>1 $$  
so that the limit of the considered tower is bigger than $1$.
\end{rem}

\subsubsection*{Some examples}
We now use Proposition \ref{lem asint} to give examples of asymptotically good AG-codes over $\ff_{q^2}$, $\ff_{q^3}$ and also over finite fields $\ff_q$ of big enough cardinality. Of course, if we require something more than the linearity of the sequence of codes given by Proposition \ref{lem asint}, the situation becomes much more delicate and restrictive (see the next Section).

\begin{exam}\label{goodag1}
For each $q>2$ there is a family of AG-codes over $\ff_{q^2}$ which is asymptotically good. To show this we just use the function fields in the recursive tower $\mathcal{F}=\{F_i\}_{i=0}^{\infty}$ over $\ff_{q^2}$ of Garcia and Stichtenoth (\cite{GS}) whose defining equation is
\begin{equation} \label{yq+y} 
y^q+y = \frac{x^q}{x^{q-1}+1}.
\end{equation}
It is known (see, for instance, \cite[Example 5.4.1]{NX}) that
$N(F_i) \geq q^{i-1}(q^2-q)+1$ and
$$ g(F_i) = \begin{cases} (q^{\frac{i}{2}}-1)^2 & \quad \text{for $i$ even,} \sk
\\ (q^{\frac{i+1}{2}}-1)(q^{\frac{i-1}{2}}-1) & \quad \text{for $i$ odd.}
\end{cases}$$
Thus, by taking $n_i=q^{i-1}(q^2-q)$, we have $n_i+1$ rational places $P_{i,1},\ldots, P_{i,n_i}, Q_i$ in $F_i$ and
$$ \frac{n_i}{g(F_i)} \geq q-1, $$
for all $i\geq 1$. We readily see that condition (a) in Proposition \ref{lem asint} holds and the same happens with condition (b)
with $\lambda = (q-1)^{-1} = A(q^2)^{-1}$. Finally, by taking $G_i=Q_i$, we see that condition (c) also holds with
$\alpha(i) = 1/n_i$. 
In this way, by Proposition \ref{lem asint},
the sequence $\{\cod_i\}_{i\in \na}$ of AG-codes $\cod_i = C(D_i,Q_i)$, with $D_i = P_{i,1} + \cdots + P_{i,n_i}$, is asymptotically good and attains the TVZ-bound over $\ff_{q^2}$ for $q\geq 3$.
\end{exam}

Serre's type lower bound for $A(q)$ give us a way to prove that the family of AG-codes over $\ff_{q}$ is asymptotically good for any $q$ large enough.

\begin{exam}\label{goodag2}
It is well known that $A(q) \ge \tfrac{1}{96} \log_2 q$, for any prime power $q$ (\cite[Theorem 5.2.9]{NX}).
By definition of $A(q)$ there exists a sequence of function fields $\{F_i\}_{i=1}^{\infty}$ over $\ff_q$  such that
$$ \frac{N_i}{g(F_i)} \ge \tfrac{1}{96} \log_2q >1, $$
for $q>2^{96}$, where $N_i=N(F_i)$. 
Take $n_i = N_i-1$ and consider $\ff_q$ with $q>2^{96}$. Thus, condition (b) in Proposition \ref{lem asint} holds with $\lambda = \tfrac{96}{\log_2 q}$ and, since $g(F_i)\rightarrow\infty$ as $i\rightarrow\infty$, we must have that $N_i\rightarrow\infty$ as $i\rightarrow\infty$ so that condition (a) also holds.
The divisor $G_i$ is simply the remaining rational place of $F_i$ after using the $N_i-1$ rational places to define $D_i$.
Clearly, condition (c) holds with $\alpha(i)=1/n_i$. Thus, by Proposition \ref{lem asint}, there is a sequence of asymptotically good AG-codes over $\ff_q$ attaining an $(\ell,\delta)$-bound for any $q>2^{96}$ with $\ell = 1-\tfrac{96}{\log_2q}$.
\end{exam}

\begin{exam}\label{goodag3}
There is a generalization of Zink's lower bound for 
$A(q^3)$ proved in \cite{BGS}, namely for any $q\geq 2$ we have
$$ A(q^3) \geq \tfrac{2(q^2-1)}{q+2} >1 \,.$$
A similar argument as in Example \ref{goodag2}
shows that for any $q$ there is an asymptotically good AG-code over $\ff_{q^3}$ attaining an $(\ell,\delta)$-bound with 
$\ell=1-\tfrac{q+2}{2(q^2-1)}$.
\end{exam}

In the previous examples we have obtained a sequence of AG-codes attaining the TVZ-bound without any additional structure other than linearity. Can one construct good AG-codes with a prescribed property such as for instance cyclicity, transitivity or self-duality? 
For self-dual and transitive codes this was done in \cite{Sti06}. The problem for cyclic codes is still open. 
We will show later that there are families of so called block-transitive codes, which we now introduce, attaining the TVZ-bound.

\section{Good block-transitive codes}  \label{blocktransitivedef}
\subsection{Block-transitive codes}
There is a natural action of the group $\mathbb{S}_n$ on $\ff_q^n$ given by  
$$ \pi \cdot v = \pi \cdot (v_1,\ldots,v_n) =(v_{\pi(1)},\ldots,v_{\pi(n)}) $$
where $\pi\in \mathbb{S}_n$ and $v=(v_1,\ldots,v_n) \in\ff_q^n$.
The set of all $\pi\in \mathbb{S}_n$ such that $\pi\cdot c\in \cod$ for all codewords $c$ of $\cod$ forms a subgroup $\mathrm{Perm}(\cod)$ of $\mathbb{S}_n$ which is called the \textit{permutation group} of $\cod$ 
(sometimes denoted by $\mathrm{PAut}(\cod)$), that is 
$$\mathrm{Perm}(\cod) = \{\pi \in \mathbb{S}_n : \pi\cdot \cod\subset\cod \}.$$
A code $\cod$ is called {\em transitive} if there is a subgroup $\Gamma$ of $\mathrm{Perm}(\cod)$ acting transitively on the coordinates of the codewords of $\cod$, i.e.\@ for 
every $1\leq i<j\leq n$ there exists $\pi\in \Gamma$ such that $\pi(i)=j$. An important subfamily of transitive codes is the class of {\em cyclic} codes. These are the codes which are invariant under the action of the cyclic shift $s\in \mathbb{S}_n$ defined by 
$s(1)=n$ and $s(i)=i-1$ for $i=2,\ldots, n$, i.e.\@ a code $\cod$ is cyclic if 
$s\cdot c =(c_{s(1)}, c_{s(2)},\ldots,c_{s(n)})=(c_n, c_1,\ldots,c_{n-1}) \in \cod$ for every $c=(c_1,c_2,\ldots,c_n)\in \cod$.

Suppose now that for a fixed $n\in \na$ we have a partition 
\begin{equation} \label{partition}
n=m_1+m_2+\cdots+m_r 
\end{equation}
where $m_1,\ldots,m_r$ are positive integers. 
There is a natural action of the direct product of the groups $\sg_{m_1} \times \cdots \times \sg_{m_r}$ on 
$\ff_q^n = \ff_q^{m_1}\times \ff_q^{m_2} \times \cdots \times\ff_q^{m_r}$, which restricts to $\cod$ 
as well, as follows. 
Consider any $v\in\ff_q^n$ divided into $r$ consecutive blocks $v_i$ of $m_i$ coordinates each, 
$$v=(v_1,\ldots,v_r)=(v_{1,1},\ldots,v_{1,m_1}, \ldots,v_{r,1},\ldots,v_{r,m_r}).$$
Then, if $\pi= (\pi_1,\ldots,\pi_r)\in \sg_{m_1}\times\cdots\times \sg_{m_r}$, we have the product action 
$$\pi \cdot v = (\pi_1 \cdot v_1 ,\ldots,\pi_r \cdot v_r) = (v_{1,\pi_1(1)},\ldots,v_{1,\pi_1(m_1)},\ldots,v_{r,\pi_r(1)},\ldots,v_{r,\pi_r(m_r)}).$$
In other words, each $\sg_{m_i}$ acts transitively, and in an independent way, on the corresponding blocks in which the words of $\ff_q^n$ are divided. 

The set of all $(\pi_1,\ldots,\pi_r)\in \sg_{m_1}\times \cdots \times \sg_{m_r}$ 
such that $(\pi_1,\ldots,\pi_r)\cdot c\in \cod$ for all codewords $c$ of $\cod$ forms a subgroup of 
$\sg_{m_1} \times\cdots\times \sg_{m_r}$. We call it the \textit{$r$-block permutation group} 
of $\cod$ and we denote it by $\mathrm{Perm}_r(\cod)$. Of course, we have 
$\mathrm{Perm}_r(\cod) \le \mathrm{Perm}(\cod)$. 

\begin{defi} \label{block codes}
Let $\cod$ be a code over $\ff_q$ of length $n$. We say  that $\cod$ is a  \emph{block-transitive} code if $n$ can be partitioned as in \eqref{partition} for some $r\in\na$  in such way that there is a subgroup 
$\Gamma$ of $\mathrm{Perm}_r(\cod)$  acting  on $\cod$ as above.
In other words, if
$c=(c_1,c_2,\ldots,c_r)  
\in \cod,$
and we consider any of its blocks, say $c_k=(c_{k,1},\ldots,c_{k,m_k})$, and two indices $1\leq i<j<m_k$, then there exists 
$\pi=(\pi_1,\ldots,\pi_r)\in \Gamma$ such that $\pi_k(i)=j$. 

If, in addition, $m_1=\cdots=m_r=m$ we say that $\cod$ is an  \emph{$r$-block-transitive} code. Moreover if $\pi_1=\cdots=\pi_r$ then we have the \textit{$r$-quasi transitive} codes defined by Bassa in his Thesis \cite{B06}. Notice that when $r=1$ we are just in the case of  transitive codes. 
\end{defi}


We will next show how to construct geometric block transitive codes. 
We first recall some standard facts about places in Galois extensions.
Let $F$ be a function field over $\ff_q$ and let $E/F$ be a finite field extension. If $Q$ is a place of $E$ we will use the standard symbol $Q | P$ to denote that $Q$ lies over the place $P$ of $F$, i.e.\@ $P=Q\cap F$.
Now, assume that $E$ is a separable extension and that $F=\ff_q(x)$ is a rational function field over $\ff_q$. It is well known (see, for instance, \cite{Sti}) that the group $\aut(E/F)$ of $\ff_q(x)$-automorphisms of $E$ acts on the set of all places of $E$ and this action is extended naturally to divisors. Even more, if $E/F$ is Galois and $\{P_1,\ldots,P_n\}$ is the set of all places of 
$E$ lying above a place $P$ of $F$, then for each $1\leq i<j\leq n$ there exists $\sigma\in \gal(E/F)$ such that $\sigma (P_i)=P_j$. 
This means that $\gal(E/F)$ acts transitively on the set $\{P_1,\ldots,P_n\}$. 

Let $C_{\rr}(D,G)$ be an AG-code as in \eqref{C(DG)}. Suppose now that $\sigma(G)=G$ for any $\sigma\in H$ where $H$ is a subgroup of $\aut(E/F)$. Then, there is an action of $H$ on 
$C_{\rr}(D,G)$ defined as 
$$ \sigma\cdot \big( u(P_1),u(P_2),\ldots,u(P_n) \big)\colon\!\!=\big( u(\sigma(P_1)),u(\sigma(P_2)),\ldots,u(\sigma(P_n)) \big). $$
It is clear that all of this implies that the AG-code $C_{\rr}(D,G)$ is transitive if $H$ acts transitively on 
the set $\{P_1,\ldots,P_n\}$. Similarly, if $H$ is a cyclic group then, under the above conditions, the AG-code $C_{\rr}(D,G)$ is cyclic.

We first show how to construct a transitive AG-code. 
Let $F$ be an algebraic function field over $\ff_q$ and let $P,Q$ be places in $F$ with $P$ rational.
Consider any finite Galois extension $E$ of $F$ of degree $n$. If $P$ splits completely in $E$, then the AG-code 
$\cod=C_\rr^E(D=P_1+\cdots+P_n,G)$, where $P_i|P$ for $1\le i \le n$ and $G=\sum_{Q'|Q} Q'$ is the divisor obtained by summing over all the places in $E$ above $Q$, is well defined since each $P_i$ is rational and disjoint from $G$. Moreover, $\cod$ is transitive, since 
$\Gamma=\gal(E/F)$ acts transitively on $D$ and $G$ (hence, in particular, fixing them).

Now, to get $r$-block transitive AG-codes we proceed as follows. 
Let $E$ be a finite Galois extension of degree $m$ of a function field $F$ over $\ff_q$. Consider $r$ rational places $P_1,\ldots,P_r$ in $F$ which split completely in $E$ 
and a divisor $Q$ disjoint with all of them. Let $T(P_i)=\{P_{i,1},\ldots,P_{i,m}\}$ be the places in $E$ above $P_i$, for each $1\le i \le r$. 
Let $$D=P_{1,1}+ \cdots + P_{1,m} + \cdots + P_{r,1}+ \cdots + P_{r,m} \qquad \text{and} \qquad G=\sum_{Q'|Q} Q'.$$ 
Thus, all $P_{i,j}$ are rational and $G$ is disjoint with $D$.
Since $G=Gal(E/F)$ acts transitively in each $T(P_i)$, $1\le i \le r$ and fixes $D$ and $G$, the AG-code $\cod=C_\rr^E(D,G)$
is an $r$-block transitive AG-code.

\subsection{Good block-transitive codes}
In this subsection we will present sufficient conditions to get good block-transitive AG-codes from towers.

We first review some basic definitions on ramification in towers of algebraic function fields, which will be used throughout the rest of the paper. The standard reference for all of this is \cite{Sti}. 
Let $F$ be a function field over a perfect field $K$ and let $E$ be a field extension of $F$ of finite degree. Let $Q$ and $P$ be places of $E$ and $F$, respectively, with $Q|P$. Also, denote as usual by $e(Q | P)$ and $f(Q | P)$ the ramification index and the inertia degree of $Q | P$, respectively. 
We say that $P$ {\em splits completely} in $E$ if $e(Q|P)=f(Q|P)=1$ for any place $Q$ of $E$ lying over $P$ 
(hence there are $[E:F]$ places in $E$ above $P$). We say that $P$ \textit{ramifies} in $E$ 
if $e(Q|P)>1$ for some place $Q$ of $E$ above $P$ and that it is 
{\em totally ramified} in $E$ if there is only one place $Q$ of $E$ lying over $P$ and $e(Q|P)=[E:F]$ (hence $f(Q|P)=1$). 
It is known that the different exponent $d(Q|P)$ satisfies $e(Q|P)-1 \le d(Q|P)$. 
If $b\in \mathbb{R}$, the extension $E/F$ is called {\em b-bounded} (see \cite{GS}) if for any place $P$ of $F$ and any place 
$Q$ of $E$ lying over $P$ we have 
$$d(Q|P) \le b(e(Q|P)-1).$$  

Let $\F=\{F_i\}_{i=0}^{\infty}$ be a tower of function fields over $\ff_q$. The \textit{ramification locus} $R(\F)$ of $\F$ is the set 
of places $P$ of $F_0$ such that $P$ is ramified in $F_i$ for some $i\geq 1$.
The \textit{splitting locus} $Sp(\F)$ of $\F$ is the set of rational places $P$ of $F_0$ such that $P$ splits completely in $F_i$ for all $i\geq 1$. Notice that the set $R(\F)$ could be empty, finite or even infinite whereas the set $Sp(\F)$ could be empty or finite but never infinite. 
A place $P$ of $F_0$ is {\it totally ramified} in the tower if for each $i\geq 1$ there is only one place $Q$ of $F_i$ lying over $P$ and 
$e(Q|P)=[F_i:F_0]$. 
We introduce the following definition. 
\begin{defi} \label{v-ram}
Given an integer $\mu>1$ and a tower $\mathcal{F}=\{F_i\}_{i=0}^{\infty}$, a place $P$ of $F_0$ is {\it absolutely $\mu$-ramified} in $\mathcal{F}$ if for each $i\geq 1$ and any place $Q$ of $F_i$ lying over $P$ we have that $e(R|Q)\geq \mu$ for any place $R$ of $F_{i+1}$ lying over $Q$. 
\end{defi}

The tower $\F$ is called \textit{tame} or {\it tamely ramified} if for any $i\geq 0$, any place $P$ of $F_i$ and any place $Q$ of 
$F_{i+1}$ lying over $P$, the ramification index $e(Q|P)$ is not divisible by the characteristic of $\ff_q$. 
Otherwise, $\F$ is called \textit{wild} or {\it wildly ramified}.
Furthermore, $\F$ has {\em Galois steps} if each extension $F_{i+1}/F_i$ is Galois. One says that $\F$ is a  {\it b-bounded tower} if each extension 
$F_{i+1}/F_i$ is a $b$-bounded Galois $p$-extension where $p=\textrm{char}(\ff_q)$.  If each extension $F_i/F_0$ is Galois, $\F$ is said to be a {\it Galois tower} over $\ff_q$.

\subsubsection*{Conditions for the existence of good block-transitive AG-codes from towers}
It is known that the class of transitive codes attains the TVZ-bound over $\ff_{q^2}$ (\cite{Sti06}), i.e.\@, 
by \eqref{Aq2}--\eqref{lAq}, an 
$(\ell,\delta)$-bound with 
\begin{equation} \label{tvz}
\ell = 1-\frac{1}{A(q^2)} = 1 - \frac{1}{q-1}. 
\end{equation}
Also, from \cite{B06}, the class of $r$-quasi transitive codes attains an $(\ell,\delta)$-bound over $\ff_{q^3}$ with 
$$\ell = 1 - \frac{q+2}{2r(q-1)}.$$
In both cases wildly ramified towers were used. We will later prove a general result for the case of $r$-block transitive codes by using both wildly and tamely ramified towers (see Section 6).

\begin{rem} 
Note that, by definition, block-transitive codes
are direct sums of transitive codes. It is clear that the direct sum of asymptotically good codes is again good. However, even though 
the family of transitive codes is asymptotically good, nothing can be said in general about the asymptotic behaviour of 
block-transitive codes, which are a direct sum of possibly different transitive codes. 
\end{rem}

We now present a general result giving sufficient conditions to have good $r$-block transitive AG-codes by using the Galois closure of a certain class of towers.

\begin{thm}\label{generalgoodag}
Let $\F=\{F_i\}_{i=0}^{\infty}$ be either a tamely ramified tower with Galois steps or a $2$-bounded tower over $\ff_q$ with non empty splitting locus 
$Sp(\F)$ and non empty ramification locus $R(\F)$. 
Suppose there are finite sets $\Gamma$ and $\Omega$ of rational places of $F_0$ such that
$R(\F)\subset \Gamma$ and 
$\Omega\subset Sp(\F)$ with  
\begin{equation}\label{keycondition}
0<g_0-1 + \epsilon t< r,
\end{equation}
where $g_0=g(F_0)$, $t=|\Gamma|$, $r=|\Omega|$ and $\epsilon = \tfrac 12$ if $\F$ is tamely ramified or $\epsilon=1$ otherwise.
If there is a place $P_0\in R(\F)$ which is absolutely $\mu$-ramified in $\mathcal{F}$ 
for some $\mu>1$, then there exists a sequence $\mathcal{G}=\{\cod_i\}_{i=0}^\infty$ of $r$-block transitive AG-codes over $\ff_q$ attaining a 
$(\ell,\delta)$-bound with
\begin{equation}\label{ellg0}
\ell=1-\frac{g_0-1+\epsilon t}{r}.
\end{equation}
In particular, the Manin's function satisfies $\alpha_q(\delta)\geq \ell-\delta$. 
Moreover, the sequence $\mathcal{G}$ is defined over the Galois closure $\mathcal{E}$ of $\mathcal{F}$ with limit $\lambda(\mathcal{E})>1$.
\end{thm}

\begin{proof} We divide the proof in several steps for clarity.

\noindent \textit{Step 1: the Galois closure $\mathcal{E}$ of $\mathcal{F}$}. 
Let $F'_i$ be the Galois closure of $F_i$ over $F_0$, which is simply the compositum of fields of the form $\sigma(F_i)$ where $\sigma$ ranges over the finitely many  $F_0$-embeddings of $F_i$ into an algebraic closure $\overline{F_0}$ of $F_0$. Denote by $\mathcal{E}=\{F'_i\}_{i=0}^{\infty}$ the tower which is the Galois closure of $\mathcal{F}$. 

From Proposition 2.1,  Theorem 2.2 and Remark 2.4 of \cite{GS} we have that 
$\mathcal{E}$ is also a tamely ramified tower of function fields over $\ff_q$ with Galois steps or a $2$-bounded tower with $\ff_q$ as the full field of constants of each $F'_i$; with the same splitting and ramification locus of $\F$, i.e. 
$Sp(\F)=Sp(\mathcal{E})$ and $R(\F)=R(\mathcal{E})$, 
and also 
\begin{equation}\label{genusgprime}
\gamma(\mathcal{E}) = \lim_{i\rightarrow\infty}\frac{g(F'_i)}{[F'_i:F_0]}\leq g(F_0)-1+\epsilon\sum_{P\in R(\mathcal{E})}\deg P\leq g_0-1+\epsilon t.
\end{equation}
On the other hand, each place of $\Omega$ also splits completely in the Galois closure $F'_i$ of $F_i$ over $F_0$, so that we have 
$n_i=r[F'_i:F_0]$ rational places of $F'_i$ 
\begin{equation} \label{Sis}
S_{i,1},S_{i,2}\ldots,S_{i,n_i}.
\end{equation}
lying over the places of $\Omega$. Thus, $\nu(\mathcal{E})\geq r$. 
Therefore, by  
hypothesis, \eqref{genusgprime} and \eqref{limit}, we get 
\[ \lambda(\mathcal{E})\geq \frac{r}{g_0-1+\epsilon t}>1.\]

\noindent \textit{Step 2: the codes over $\mathcal{E}$}. 
Now we define a family of AG-codes associated to the tower $\mathcal{E}$. For each $i\ge 1$ let us consider the AG-code 
\begin{equation} \label{the codes}
\cod_i = C_{\rr}(D_i,r_iG_i),
\end{equation}
with $r_i\in \na$,
\begin{equation} \label{the divisors}
D_i=S_{i,1} +\cdots+ S_{i,n_i} \qquad \text{and} \qquad G_i=R_{i,1} + \cdots + R_{i,k_i}.
\end{equation}
Here the rational places $S_{i,j}$ are the ones given in \eqref{Sis} while $R_{i,1}, \ldots, R_{i,k_i}$ are all the places of $F_i'$ 
lying over $P_0$. 
We will show later (see Remark \ref{bound for r's}) an upper bound for the integers $r_i$, whose existence is given by Proposition~\ref{lem asint}.

\noindent \textit{Step 3: induction on the degree of the places}. 
Let $T_1$ be the set of places of $F_1$ lying over $P_0$, which is rational by hypothesis. Since $F_1/F_0$ is Galois and $P_0$ is absolutely 
$\mu$-ramified in the tower, for every place  $P\in T_1$ we have that $P|P_0$ is ramified with ramification index $e=e(P|P_0)\geq \mu$ and relative degree $f=f(P|Q)$ so that 
$\mathit{def} = [F_1:F_0]$
where $d=|T_1|$. Then
\[ \deg  \sum_{P\in T_1} P  = df = \frac{[F_1:F_0]}{e}\leq \frac{[F_1:F_0]}{\mu}\,.\]
Now suppose that
\[\deg \sum_{P\in T_{i-1}} P \leq \frac{[F_{i-1}:F_0]}{\mu^{i-1}}\,,\]
where $T_{i-1}$ is the set of places of $F_{i-1}$ lying over $P_0$. Let $P\in T_{i-1}$ and let $R_1,\ldots,R_{r_P}$ be all the places of $F_i$ lying over $P$. Since $F_i/F_{i-1}$ is a Galois extension, every place $R_j$ over $P$ is ramified with  the same ramification index $e_P$ and relative degree $f_P$. Thus $r_P \, e_P \, f_P = [F_i:F_{i-1}]$ and 
\[ r_P \, f_P = \frac{[F_i:F_{i-1}]}{e_P} \leq \frac{[F_i:F_{i-1}]}{\mu}\,.\]
Hence, by inductive hypothesis, we have
\begin{align*}
\deg \sum_{R\in T_i} R  & = \sum_{P\in T_{i-1}}r_P f_P \deg P \leq \frac{[F_i:F_{i-1}]}{\mu} \,\sum_{P\in T_{i-1}}\deg P  \\
&\leq \frac{[F_i:F_{i-1}]}{\mu}\frac{[F_{i-1}:F_0]}{\mu^{i-1}} = \frac{[F_i:F_0]}{\mu^i} \,,
\end{align*}
where $T_i$ is the set of places of $F_i$ lying over $P_0$. We have proved that if $T_i$ is the set of places of $F_i$ lying over $P_0$ then 
\begin{equation}\label{degreeT_i}
\deg \sum_{P\in T_i} P \leq \frac{[F_i:F_0]}{\mu^i},
\end{equation}
for all $i\in \na$.

\noindent \textit{Step 4: estimating the degree of $G_i$}. 
We can now give an upper bound for the degree of the divisors $G_i$. Note that $\sop G_i \cap F_i=T_i$ and if we denote by $R_{i,1}^P,\ldots, R_{i,n_P}^P$  all the places in $\sop G_i$ lying over a place $P\in T_i$,  we have that
\begin{align}\label{bound-deg-Gi}
\begin{split}
\deg G_i & = \sum_{j=1}^{k_i} \deg R_{i,j} = \sum_{P\in T_i} \sum_{t=1}^{n_P} \deg R_{i,t}^P \\
& =\sum_{P\in T_i} \sum_{t=1}^{n_P} f(R_{i,t}^P|P) \, f(P|P_0) = \sum_{P\in T_i} f(P|P_0) \sum_{t=1}^{n_P} f(R_{i,t}^P|P) \\
& \leq \sum_{P\in T_i} f(P|P_0)[F'_i:F_i] = [F'_i:F_i] \, \deg  \sum_{P\in T_i} P \leq \frac{[F'_i:F_0]}{\mu^i}\,,
\end{split}
\end{align}
where in the last inequality we have used \eqref{degreeT_i}.

\noindent \textit{Step 5: obtaining the $(\ell,\delta)$-bound}. 
We will now apply Proposition \ref{lem asint} to the tower $\mathcal{E}$ and the sequence of codes $\cod$ given in \eqref{the codes} and 
\eqref{the divisors}. Since $n_i=r[F'_i:F_0]$, condition (a) in the proposition holds. Also, 
by \eqref{genusgprime} and the hypothesis \eqref{keycondition}, we have
$$ \frac{g(F'_i)}{n_i} \leq \frac{g_0-1+\epsilon t}{r}<1$$
for $n_i$ big enough and, hence, condition (b) holds with $\lambda=\tfrac{1}{r}(g_0-1+\epsilon t) <1$.
Furthermore, from the above upper bound \eqref{bound-deg-Gi} for $\deg G_i$ we also have 
\[\frac{\deg G_i}{n_i}\leq \frac{1}{\mu^ir}, \]
so that condition (c) holds by taking $\alpha(i) = 1/\mu^ir$. In this way, by Proposition \ref{lem asint}, the sequence
$\{\cod_i\}_{i=1}^{\infty}$ of AG-codes
$\cod_i = C_{\rr}(D_i,r_iG_i)$, for certain integers $r_i>0$ provided by the proof of the proposition, 
is asymptotically good over $\ff_q$ attaining a $(\ell,\delta)$-bound with
\[ \ell =1- \frac{g_0-1+\epsilon t}{r}.\]

\noindent \textit{Final step: block transitivity}.
Finally, notice that both divisors $D_i$ and $G_i$ are invariant under $\mathrm{Gal}(F'_i/F_0)$ so that by the definition of $D_i$ and the action of $\mathrm{Gal}(F'_i/F_0)$ on the places in the support of $D_i$, we see at once that $C_{\rr}(D_i,r_i G_i)$ is an $r$-block transitive AG-code over $\ff_q$.	
\end{proof}

Now we give a lower bound for the numbers $r_i$ in \eqref{the codes}, in terms of the parameters given in the hypothesis and the parameter 
$\delta$, which represents the desired relative minimum distance of the codes $\cod_i=C(D_i,r_iG_i)$ in the sequence (as can be seen in the proof of Proposition \ref{lem asint}). 
We have
\begin{equation} \label{ris}
r_i  \ge (1-\delta)r\mu^i-1
\end{equation}
and hence they are bounded by $r$, $\delta$ and $\mu$ for every $i$. 

To see this, without loss of generality we can take $0<\delta<1-q^{-1}$, since Manin's function $\alpha_q(\delta)=0$ for $1-q^{-1}\leq \delta<1$. Let $N=\log_{\mu}(\tfrac qr)$. Then,  for $i\ge N$ we have
$\alpha(i) = \frac{1}{r\mu^i} \le \frac{1}{q}< 1-\delta$,
and thus \eqref{etaerrei} tell us that 
\begin{equation}\label{corodela3.2}
\Big \lceil (1-\delta-\tfrac{1}{\mu^ir}) \frac{n_i}{\deg G_i} \Big \rceil \le r_i \le 
\Big \lfloor (1-\delta) \frac{n_i}{\deg G_i} \Big \rfloor.
\end{equation}
Therefore, for $i\ge N$ we get
$r_i\geq (1-\delta- \tfrac{1}{\mu^ir})\frac{n_i}{\deg G_i}\geq (1-\delta)r\mu^i-1$.

\begin{rem} \label{bound for r's}
At first, it seems that the condition of having an absolutely $\mu$-ramified place in the tower $\mathcal{F}$ in Theorem 
\ref{generalgoodag} is only for technical reasons, but the above lower bound shows that $\mu$ has something to do also with the AG-codes constructed from $\mathcal{F}$.
\end{rem}

\section{Good block-transitive codes attaining the TVZ-bound}
As a consequence of Theorem \ref{generalgoodag}, we will next show that the TVZ-bound can be attained by sequences of certain block-transitive codes (constructed from towers of function fields) whose lengths can be computed explicitly while their minimum distances and dimensions can be estimated in terms of the cardinality of the considered finite field $\ff_q$. The structure of the Galois groups acting on the codes is also known.

\subsubsection*{From wildly ramified towers}
We will construct good $r$-block transitive codes over finite fields of square cardinality $\ff_{q^2}$, 
with $r=q^2-q$. We will need the following notation; let $q>2$ and put  
\begin{equation} \label{mis}
m_i=\left\{ \begin{array}{ll}
q^{2i-1} \quad \text{(resp. $q^{2i-1-\lfloor i/2\rfloor})$} & \qquad \text{if $1\leq i\leq 2$ and $q$ is odd (resp.\@ even),} \msk \\
q^{3i-3} \quad \text{(resp. $q^{3i-3-\lfloor i/2\rfloor})$}  & \qquad \text{if $i\geq 3$ and $q$ is odd (resp.\@ even),} 
\end{array} \right. 
\end{equation}   
for each $i\in \na$.

\begin{thm} \label{goodag-psquare}
Let $q>2$ be a prime power. Then, there exists a sequence $\mathcal{G} = \{\cod_i\}_{i=1}^{\infty}$ of $r$-block transitive codes over 
$\ff_{q^2}$, with $r=q^2-q$, attaining the TVZ-bound.
Each $\cod_i$ is an $[n_i=rm_i,k_i,d_i]$-code, 
where $m_i$ is as in \eqref{mis}.
By fixing $0<\delta<1-q^{-2}$, 
we also have that 
\begin{equation} \label{kidi}
d_i\geq \delta \,(q^2-q)m_i \qquad \text{and} \qquad  
k_i \geq \{(1-\delta)(q^2-q) - (q + q^{-i})\} m_i
\end{equation}
for each $i\geq 1$, where the second inequality is non trivial for
$0<\delta< 1- \tfrac 1r (q+q^{-i})$.
\end{thm}

\begin{proof} 
Let us consider now the wildly ramified tower $\F=\{F_i\}_{i=0}^{\infty}$ over $\ff_{q^2}$ recursively defined by the equation
\[ y^q+y = \frac{x^q}{x^{q-1}+1}, \]
whose optimality was first proved in \cite{GS96} (see also Section 7.4 of \cite{Sti}). 
From these works we know that $\F$ is a $2$-bounded tower over $\ff_{q^2}$, the pole $P_{\infty}$ of $x_0$ in $F_0=\ff_{q^2}(x_0)$ 
is totally ramified in $\F$, so that $P_{\infty}$ is absolutely $q$-ramified in $\F$, at least $q^2-q$ rational places of 
$\ff_{q^2}$ split completely in $\F$ and 
the ramification locus $R(\F)$ has at most $q+1$ elements, i.e.\@ $q^2-q\le |Sp(\F)|$ and $|R(\F)| \le q+1$.

Thus, we are in the conditions of Theorem \ref{generalgoodag} with 
$$\mu=q, \qquad \epsilon=1, \qquad g_0=0, \qquad r=q^2-q \qquad \text{and} \qquad t=q+1.$$
Therefore, there exists a sequence $\mathcal{G}=\{\cod_i\}_{i\in\na}$ of $r$-block transitive AG-codes over $\ff_{q^2}$ attaining a $(\ell, \delta)$-bound with 
$$ \ell = 1-\frac{g_0-1+t}{r} = 1-\frac{q}{q^2-q} = 1-\frac{1}{q-1}, $$
which means that this sequence of codes over $\ff_{q^2}$ attains the TVZ-bound (see \eqref{tvz}).
Notice that this can also be checked directly since, from \eqref{kidi}, we have $\delta_i+R_i\ge 1-\tfrac{1}{q-1}$ for every $i$.
 
\sk 
The codes in $\mathcal{G}$ are of the form $\cod_i=C_\rr(D_i,r_iG_i)$ as in \eqref{the codes}. From \cite[Corollary 27]{BB11} we have that the degree of the extension $F'_i/F_0$ is given by the number $m_i$ in \eqref{mis}, where $F'_i$ is the Galois closure of 
$F_i$ over $F_0$. Thus, from the proof of Theorem~\ref{generalgoodag}, for each 
$i\geq 1$ we have  
$$n_i=r[F'_i:F_0]=(q^2-q)m_i=rm_i.$$ 
Now, fix $0<\delta<1-q^{-2}$. Since $\mu=q$, the right hand side of \eqref{corodela3.2} is valid for $i\geq 1$, 
because $\log_q(\tfrac{q^2}{r}) = \log_q (\tfrac{q}{q-1}) = 1 - \log_q (q-1) <1$, and thus
$$d_i \ge n_i-r_i \deg G_i \geq n_i-(1-\delta)n_i = \delta n_i$$ 
for $i\geq 1$, as desired. 

Finally, from \cite[Lemma 7.2.3]{Sti} we know that the sequence of rational numbers
$\big\{ \frac{g(F'_i)-1}{[F'_i:F_0]} \big\}_{i=0}^{\infty}$
is monotonically increasing and its limit is the genus $\gamma(\mathcal{E})$ of the tower $\mathcal{E}=\{F'_i\}_{i=0}^{\infty}$. Also, from the proof of Theorem \ref{generalgoodag} we have an upper bound for the genus $\gamma(\mathcal{E})$ of $\mathcal{E}$. Therefore 
\[\frac{g(F'_i)-1}{[F'_i:F_0]} \leq \gamma(\mathcal{E})\leq g_0-1+\varepsilon t=q, \]
so that $g(F'_i)\leq q m_i+1$ for $i\geq 0$. 
From \eqref{etaerrei} we also have
\[ r_i\deg G_i\geq (1-\delta-\tfrac{1}{rq^i}) r m_i, \]
for $i\geq 1$. Then, since $k_i \ge r_i\deg G_i+1-g(F'_i)$,  
for every $i\geq 1$ we have
$$k_i \geq (1-\delta-\tfrac{1}{rq^i}) r m_i +1-(qm_i+1) = \{ (1-\delta)r-(q+q^{-i})\} m_i.$$
It is easy to check that $(1-\delta)r-(q+q^{-i})>0$ 
if $\delta< 1- \tfrac 1r (q+q^{-i})$. 
This proves the theorem.
\end{proof}

\begin{exam}
By the previous theorem with $q=3$, there are good sequences of $6$-block transitive codes over $\ff_9$ attaining a TVZ-bound with $0<\delta<\tfrac 12$. Similarly, taking $q=4$ there are good $12$-block transitive codes over $\ff_{16}$ 
with $0<\delta<\tfrac 23$. 
\end{exam}

\subsubsection*{From tame towers}
Here we construct good $r$-block transitive codes over $\ff_{p^2}$, with $r=2(p-1)$.
For $i\in \na$ we put 
\begin{equation} \label{misss} 
m_i = \left\{ \begin{array}{ll}
2^i           & \qquad \text{if $i=1,2,3$,} \sk \\
2^{2i-3}      & \qquad \text{if $i=4,5$,} \sk \\
2^{3i-8}      & \qquad \text{if $i\geq 6$.}
\end{array} \right. 
\end{equation}

\begin{thm} \label{goodag psquare2}
Let $p\geq 13$ be a prime and put $r=2(p-1)$. 
Then, there exists a sequence $\mathcal{C}=\{\cod_i\}_{i=1}^{\infty}$ of $r$-block transitive AG-codes over $\ff_{p^2}$ attaining the 
TVZ-bound.
Each $\cod_i$ is an $[n_i=rm_i,k_i,d_i]$-code with 
with $m_i$ as in \eqref{misss}. 
Moreover, by fixing $0<\delta<1-p^{-2}$, we have  
\begin{equation} \label{kidis} 
d_i \ge 2\delta (p-1)m_i \qquad \text{and} \qquad  
k_i \ge \{2 (1-\delta)(p-1) - (2+2^{-i})\} m_i
\end{equation}
for $i\geq \log_2(p)$, where the second inequality is non trivial if $\delta< 1- \tfrac 1r (2 + 2^{-i})$.
\end{thm}

\begin{proof}
Let $q=p^2$ where $p$ is an odd prime number. The tamely ramified tower 
$\F=\{F_i\}_{i=0}^{\infty}$ over $\ff_q$ recursively defined by the equation
$$ y^2 = \frac{x^2+1}{2x}, $$
was shown to be optimal over $\ff_q$ in \cite{GS2}.   From \cite[Proposition 5.3]{GS2} we have that the ramification locus $R(\F)$ consists of exactly $6$ elements and also the pole $P_{\infty}$ of $x_0$ in $F_0$ is totally ramified in the tower.  
Also, from \cite[Theorem 5.5]{GS2} and \cite[Proposition 5.6]{GS} we have that there exists a set $\Omega\subset Sp(\F)$ of cardinality $2(p-1)$. Thus we are in the conditions of Theorem \ref{generalgoodag} with 
$$\mu=2, \qquad \varepsilon=\tfrac 12, \qquad g_0=0, \qquad r=2(p-1) \qquad \text{and} \qquad t=6.$$ 
Hence, there exists a sequence of $r$-block transitive codes over $\ff_q$ attaining a $(\ell, \delta)$-bound with 
$$ \ell = 1-\frac{2g_0-2+t}{2r} = 1-\frac{4}{4(p-1)} = 1-\frac{1}{p-1},$$
which means that this sequence of codes over $\ff_q$ attains the TVZ-bound. 
Again, this can be easily checked directly from \eqref{kidis}, since $\delta_i+R_i\ge 1-\tfrac{1}{p-1}(1+\tfrac{1}{2^{i+1}})$ for every $i$.

On the other hand, from \cite[Theorem 6.1]{BB} we have that if $p\geq 13$ the number $m_i$ in \eqref{misss} is the degree of the extension 
$F'_i/F_0$, where $F'_i$ is the Galois closure of $F_i$ over $F_0$. From the proof of Theorem \ref{generalgoodag} we have that 
$n_i=r[F'_i:F_0]=2(p-1)m_i$ for $i\geq 1$ as stated. 

Now, fix  $0<\delta<1-q^{-1}$. Since $\mu=2$  the right hand side of 
\eqref{corodela3.2} is valid for $i\geq \log_2 (p)$ because 
$$ \log_2(\tfrac qr) = \log_2(\tfrac{p^2}{2(p-1)}) = \log_2(p) + \log_2(\tfrac{p}{p-1}) - 1 < \log_2 (p). $$  
Thus, for $i\geq \log_2(p)$ we have that 
$$ d_i \geq n_i-r_i \deg G_i \geq n_i-(1-\delta)n_i =\delta n_i, $$ 
as desired. Finally, as in the proof of Theorem \ref{goodag-psquare}, we have 
$\frac{g(F'_i)-1}{[F'_i:F_0]}\leq\gamma(\mathcal{E})\leq g_0-1 + \varepsilon t=2$ for $i\geq 1$,
where $\mathcal{E}=\{F'_i\}_{i=0}^{\infty}$. Then, $g(F'_i)\leq 2 m_i+1$ for $i\geq 1$ and, from \eqref{etaerrei}, 
we also have
$$ r_i\deg G_i\geq (1-\delta- \tfrac{2^{-i}}{r})rm_i, $$
for $i\geq\log_2(p)$. Therefore, since $k_i \ge r_i\deg G_i+1-g(F'_i)$, for $i\geq \log_2(p)$ we have
$$k_i \geq (1-\delta-\tfrac{2^{-i}}{r})rm_i+1-2m_i-1 = \{2(1-\delta)(p-1)-2-2^{-i}\} m_i.$$
It is easy to check that $2(1-\delta)(p-1)-2-2^{-i}>0$ if $\delta < 1- \tfrac 1r(2+2^{-i})$,
and the result follows.
\end{proof}

\begin{exam}
By the previous theorem with $p=13$, there are good sequences of $24$-block transitive codes over $\ff_{13^2}$ attaining a TVZ-bound with 
$0< \delta < \tfrac{11}{12}$. Also, by taking $p=17$, there are good sequences of $32$-block transitive codes over $\ff_{17^2}$ attaining a TVZ-bound with 
$0< \delta < \tfrac{15}{16}$. 
\end{exam}

Note that the dimensions of the codes in the good sequences of the previous theorems are in geometric progressions. In Theorem \ref{goodag-psquare}, 
the common ratio is $q^3$ if $q$ is odd (if $q$ is even, the ratio $\tfrac{n_{i+1}}{n_i}$ oscilates between $q^3$ and $q^2$ depending on the parity of $i$) while in Theorem \ref{goodag psquare2} it is $2^3=8$.

\section{Good block-transitive codes from class field towers}
Now we prove the existence of good $r$-block transitive codes over finite fields, not necessarily of quadratic cardinality. This result will also allow us to give examples of asymptotically good $r$-block transitive codes for  many values of $r$ different from the ones given in Theorems \ref{goodag-psquare} and \ref{goodag psquare2}. The price we must pay is that of having weaker bounds for the performance of the codes.

We will need the following notations. Given $n,m \in \mathbb{Z}$, we define the auxiliary function
\begin{equation} \label{epsilon}
\varepsilon_n(m) = \left\{ 
                    \begin{array}{ll} 
                         1 & \quad \text{if $n \mid m$,} \sk \\ 
                         0 & \quad \text{if $n \nmid m$.} 
                    \end{array} \right.
\end{equation}
Also, for a polynomial $h \in \ff_q[t]$, we define
\begin{equation}
\begin{aligned}
S_q^2(h) & = \{ \beta\in\ff_q : \text{$h(\beta)$ is a non zero square in $\ff_q$}\}, \\
S_q^3(h) & = \{ \beta\in\ff_q : \text{$h(\beta)$ is a non zero cube in $\ff_q$}\}.
\end{aligned}
\end{equation}

\begin{thm} \label{goodAGanychar}
Let $h\in \ff_q[t]$ be a monic and separable polynomial of degree $m$ such that it splits completely into linear factors over $\ff_q$. 

\sk (a) Let $q$ be odd. Suppose there is a set $\Sigma_o \subset S_q^2(h)$ 
such that $u=|\Sigma_o|>0$ and 
\begin{equation} \label{sigmao1}
    2\sqrt{2u} \le  m-(u+2+\varepsilon_2(m)) < 3u.
\end{equation}
	Then, there exists a tamely ramified Galois tower $\F$ over $\ff_q$ with limit
$\lambda(\F) \geq \frac{4u}{m-2-\varepsilon_2(m)} > 1$.
In particular, there exists a sequence of asymptotically good $2u$-block transitive codes over $\ff_q$, constructed from $\mathcal{F}$, attaining an 
	$(\ell,\delta)$-bound, with $\ell = 1- \frac{m-2-\varepsilon_2(m)}{4u}$.
	
\sk (b) Let $q$ be even of the form $q=2^{2s}$. 
Suppose there is a set 
$\Sigma_e \subset S_q^3(h)$ such that $v=|\Sigma_e|>0$ and
	\begin{equation}\label{sigmaed=3}		
2\sqrt{3v} \le  m-(v+2+\varepsilon_3(m)) < 2v-\tfrac 12.
\end{equation} 
		Then, there exists a tamely ramified Galois tower $\F'$ over $\ff_q$ with limit
$\lambda(\F') \geq \frac{6v}{2(m-\varepsilon_3(m))-3} > 1$.
	In particular, there exists a sequence of asymptotically good $3v$-block transitive codes over $\ff_q$, 
	constructed from $\mathcal{F}'$, attaining an 
	$(\ell,\delta)$-bound with $\ell=1-\frac{2m-3-2\varepsilon_3(m)}{6v}$. 
\end{thm}

\begin{proof}
Let $x$ be a transcendental element over $\ff_q$ and let $K=\ff_q(x)$ the rational field over $\ff_q$.
We will use the extensions of $K$ given by the super-elliptic equation 
\begin{equation} \label{superel}
y^n=h(x)
\end{equation} 
with $n=2$ or $3$, depending on whether $q$ is odd or even, respectively. 

We now divide the proof in three parts: (i) and (ii) consider the cases of $q$ odd and even respectively, while in (iii) 
we prove items (a) and (b) of the theorem.

\sk 
(i) Let $q$ be odd. 
From our assumptions and Proposition 6.3.1 in \cite{Sti} we have that the equation \eqref{superel} with $n=2$ 
defines a cyclic Galois extension $F_o=\ff_q(x,y)$ of degree 2 of $K$ with $\ff_q$ as the full field of constants of $F_o$. Also,  
the rational places $P_1,\ldots,P_m$ of $K$, defined by the linear factors of $h(x)$, are totally ramified in $F_o/K$ and no other places than 
$P_1,\ldots,P_m$ and $P_\infty$ ramify in $F_o/K$, 
where $P_\infty$ denotes the pole of $x$. 
Moreover, for every place $Q_\infty$ in $F$ over $P_\infty$ we have $e(Q_\infty|P_\infty)=\tfrac 2d$ with $d=(2,m)$ and the genus of 
$F_o$ is given by $g(F_o)=\tfrac 12 (m-d)$. Thus, if $m$ is even then 
$e(Q_\infty|P_\infty)=1$, hence there are exactly $m$ (totally) ramified places of $K$ in $F_o$ and $g(F_o)=\tfrac{m-2}{2}$. 
On the other hand, if $m$ is odd then we have $m+1$ (totally) ramified places in $F_o/K$ and $g(F_o)=\tfrac{m-1}2$.
	
Let $P_{\beta}$ be the zero of $x-\beta$ in $K$ with $\beta\in\Sigma_o$. Then, the residual class of $h(x)$ mod $P_\beta$ is 
$$ h(x)(P_{\beta}) = h(\beta) = \gamma^2, $$
for some $\gamma\in\ff_q^*$ by hypothesis, and we have that the separable polynomial 
$$t^2-\gamma^2 = (t-\gamma)(t+\gamma)$$ 
in $\ff_q[t]$ corresponds to the reduction mod $P_{\beta}$ of of $y^2-h(x)=0$. 
By Kummer's theorem \linebreak \cite[Theorem 3.3.7]{Sti} we see that $P_{\beta}$ splits completely into $2$ rational places of $F_o$ for each $\beta\in\Sigma_o$.
	
Now pick a rational place $P_0$ of $K$ defined by one of the linear factors of $h(x)$ and let $Q_0$ be the only place of $F_o$ lying above $P_0$, let $Q_1,\ldots, Q_{2u}$ be the rational places of $F_o$ lying over the places $P_{\beta}$ with $\beta\in\Sigma_o$ and put
$$T_o=\{Q_0\} \qquad \text{and} \qquad S_o=\{Q_1,\ldots,Q_{2u}\}.$$

\sk 
(ii)	Suppose now that $q=2^{2s}$.    
This condition is equivalent to have $q=2^m$ and $q\equiv 1$ mod $3$, which in turn is equivalent to the condition 
$3 \mid 2^m-1$. But this last condition holds if and only if $m$ is even, since $2^m-1 = 2^{m-1} + 2^{m-2} + \cdots + 2^1+2^0$ and 
$3$ divides the sum of two consecutive powers of $2$ because $2^h+2^{h-1}=3\cdot 2^{h-1}$. 

Thus, from our assumptions and \cite[Proposition 6.3.1]{Sti}, we have that the equation \eqref{superel} with $n=3$
defines a cyclic Galois extension $F_e=\ff_q(x,y)$ of degree 3 of $K$, with constant field $\ff_q$. 
Furthermore, any place $Q_\infty$ in $F_e$ over $P_\infty$ satisfies $e(Q_\infty|P_\infty)=\tfrac 3d$ with $d=(3,m)$ and the genus of 
$F_e$ is given by $g(F_e) = (m-1)- \tfrac{d-1}{2}$. Thus, if $3\mid m$ then $d=3$ and $e(Q_\infty|P_\infty)=1$, hence there are exactly $m$ (totally) ramified places of $K$ in $F_e$ and $g(F_e)=m-2$. On the other hand, if $3\nmid m$ then $d=1$ and hence we have 
$m+1$ (totally) ramified places in $F_e/K$ and $g(F)=m-1$.
	 
Let $P_{\beta}$ be the zero of $x-\beta$ in $\ff_q(x)$ where $\beta\in\Sigma_e$. Then the residual class of $h(x)$ mod $P_\beta$ is
	$$ h(x)(P_{\beta}) = h(\beta) = \gamma^3,$$
	for some $\gamma\in\ff_q^*$ by hypothesis, and we have that the separable polynomial 
     $$t^3-\gamma^3 = (t-\gamma) (t-\omega \gamma) (t- \omega^2 \gamma) $$ 
in $\ff_q[t]$  (where $\omega$ is a primitive $3$-rd root of unity) corresponds to the reduction mod $P_{\beta}$ of \linebreak
$y^3-h(x)=0$. By Kummer's theorem \cite[Theorem 3.3.7]{Sti} we see that $P_{\beta}$ splits completely into 3 rational places of $F_e$ for each $\beta\in\Sigma_e$.
	
	Again, pick a rational place $P_0$ of $K$ defined by one of the linear factors of $h(x)$ and let $R_0$ be the only place of $F_e$ lying above $P_0$ 
     and $R_1,\ldots, R_{3v}$ be the rational places of $F_e$ lying over the places $P_{\beta}$ with $\beta\in\Sigma_e$ and put
	$$T_e=\{R_0\} \qquad \text{and} \qquad S_e=\{R_1,\ldots,R_{3v}\}.$$
	
\sk
(iii) Let $\Sigma$ and $F$ be either $\Sigma_o$ and $F_o$ or $\Sigma_e$ and $F_e$, depending on the parity of $q$. Then,  
$$\mathcal{R}(F/K):=\{ P\in \mathbb{P}(K) : P \text{ ramifies in } F\}$$ 
has cardinality $m$ or $m+1$ depending on the parities of $q$ and $m$. 
More precisely, 
$$\# \mathcal{R}(F/K) = m+1-\varepsilon_n(m),$$
	where $n=2$ if $q$ is odd and $n=3$ if $q$ is even. 
	Therefore, 
     \[ \#\mathcal{R}(F/K) \ge 2 + |\Sigma| + 2\sqrt{n|\Sigma|}\]
	always holds, by \eqref{sigmao1} and \eqref{sigmaed=3}. 
          Therefore, the $T$-tamely ramified and $S$-decomposed Hilbert tower 
     $\mathcal{H}_S^T$ of $F$ is infinite, by Corollary 11 in \cite{AM}, where $T=T_o$ and $S=S_o$ if $q$ is odd or $T=T_e$ and $S=S_e$ otherwise. 
	
	This means that there is a sequence $\mathcal{F} = \{F_i\}_{i=0}^{\infty}$ of function fields over $\ff_q$ such that  $F_0=F$, 
     $\mathcal{H}_S^T = \bigcup_{i=0}^{\infty}F_i$ and that, for any $i\geq 1$, each place in $S_o$ (resp.\@ $S_e$) splits 
completely in $F_i$ (hence $\ff_q$ is the full constant field of $F_i$), the place $Q_0$ (resp.\@ $R_0$) is tamely and absolutely ramified in the tower, 
$F_i/F_{i-1}$ is an abelian extension with $[F_i:F]\rightarrow\infty$ as $i\rightarrow\infty$ and $F_i/F_0$  is unramified outside $T$. In particular, 
$T=R(\mathcal{H}_S^T)$ is the ramification locus of the tower 
     $\mathcal{H}_S^T$. 
     
     Then, we are in the situation of Theorem \ref{generalgoodag} with $F_0=F$, $\Gamma=T$ and $\Omega=S$. Let us check that 
		$g(F)-1+\tfrac 12 |\Gamma|=g(F)-\tfrac 12 <|\Omega|$ holds in each case. 
		We have seen before that 
$$g(F_o)=\tfrac 12 \{m-1-\varepsilon_2(m)\} \qquad \text{and} \qquad g(F_e)=m-1-\varepsilon_3(m).$$ 
Thus,
by the hypothesis \eqref{sigmao1} and \eqref{sigmaed=3}, we have 				
\[g(F)-\tfrac 12 = \left\{\begin{array}{cll}
\tfrac 12 \{m-2-\varepsilon_2(m)\} & <  2u = |S_o|, & \qquad \text{if $q$ is odd}, \\[3mm]
m-\tfrac 32 - \varepsilon_3(m)     & < 3v = |S_e|, & \qquad \text{if $q$ is even}.
\end{array} \right. \]
	In this way, equation \eqref{keycondition} holds and the conclusion in each case thus follows. 
Finally, note that $\lambda(\mathcal{F}) \ge |S|/\{g(F)-1+\tfrac 12\} >1$ as we wanted to see, and also that by \eqref{ellg0} we get the desired expressions for $\ell$ in (a) and (b) of the statement, respectively.
\end{proof}

The case $q=2^{2s+1}$, not covered in Theorem \ref{goodAGanychar}, cannot be handled with the method used in the proof. This is because the class field tower obtained in \cite{AM} is asymptotically good provided the first step $F_0/K$ is cyclic and for this we need a primitive $3$-rd root of unity in $K$.

We now give an example to show how Theorem \ref{goodAGanychar} works.
\begin{exam} \label {ej q odd}
(i) Let $\ff_{25}=\ff_2(\alpha)$ where $\alpha^2+4\alpha+2=0$.
Since $q$ is odd, \eqref{sigmao1} takes the form
$$2\sqrt{2u}+u+2 \le m < 4u+2.$$
Consider the polynomial $h(x)=(x-a_1)(x-a_2)\cdots(x-a_9) \in \ff_{25}[x]$
whose roots $a_i$ are given by $2, \, \alpha, \, \alpha +3, \, 2\alpha, \, 2\alpha+1, \, 2\alpha+2, \, 3\alpha +1, \, 4\alpha+1$ and $4\alpha+3$.
It is easy to check, using the software \textsc{Sage} for instance, that 
$$ S_{25}^2(h) = \{ \beta\in\ff_{25}\,:\, \text{$h(\beta)$ is square in $\ff_{25}^*$}\} = \{1,3,\alpha+1,3\alpha+4\}.$$
Thus, since $m=9$, \eqref{sigmao1} only holds with $u=2$. 
Therefore, by (a) of Theorem \ref{goodAGanychar} there exists a sequence of $4$-block transitive AG-codes over 
$\ff_{25}$ attaining an $(\ell,\delta)$-bound with $\ell=1-\tfrac 68 = \tfrac 14$, whose associated tamely ramified Galois tower 
$\mathcal{F}$ has limit $\lambda(\mathcal{F}) \ge \tfrac 87 = 1.14285714\ldots >1$.

\sk 
(ii) Now, let $\ff_{64}=\ff_{2^6}=\ff_2(\alpha)$ where $\alpha^6+\alpha^4+\alpha^3+\alpha+1=0$ and consider the polynomial 
$h(x) = (x-b_1)(x-b_2)\cdots (x-b_{21})$ of degree 21 over $\ff_{64}$ whose roots $b_i$ are  
\begin{gather*}
0,  \: \alpha^2, \: \alpha^2 + \alpha + 1, \: \alpha^3, \: \alpha^3 + 1, \: \alpha^3 + \alpha^2 + \alpha, \: \alpha^4,  
\alpha^4 + \alpha, \: \alpha^4 + \alpha + 1, \alpha^4 + \alpha^2 + \alpha + 1, \\ 
\alpha^4 + \alpha^3 + \alpha + 1, \alpha^5, \: \alpha^5 + \alpha^2, \: \alpha^5 + \alpha^2 + \alpha, \: \alpha^5 + \alpha^3 + 1, \: \alpha^5 + \alpha^4 + \alpha^2 + 1, \: \alpha^5 + \alpha^4 + \alpha^2 + \alpha, \\ 
\alpha^5 + \alpha^3 + \alpha^2 + \alpha, \: \alpha^5 + \alpha^4 + \alpha^3 + 1, \:  \alpha^5 + \alpha^4 + 1, \: \alpha^5 + \alpha^4 + \alpha^2 + \alpha + 1.
\end{gather*}
Again, using \textsc{Sage}, one checks that 
$ \#S_{64}^3(h) = \# \{ \beta\in\ff_{64}\,:\, \text{$h(\beta)$ is a cube in $\ff_{64}^*$}\} = 14$.
For $q$ even and $m=21$ expression \eqref{sigmaed=3} reads 
$$2\sqrt{3v}+v+3\le 21 < 3v+\tfrac 52$$
and holds for $v=7$ and $v=8$. 
Therefore, by (b) in the theorem and using $v=8$ there is a sequence of $24$-block transitive AG-codes over 
$\ff_{64}$ attaining an $(\ell,\delta)$-bound with $\ell=\tfrac{9}{48}$. The limit of the associated tamely ramified Galois tower 
$\mathcal{F}'$ has limit $\lambda(\mathcal{F}')\ge 1+\tfrac{11}{37} = 1,297\ldots >1$. 
\end{exam}

We will now show, as a corollary of the previous theorem, that there are asymptotically good $N$-block transitive 
codes over $\ff_q$, with $N$ a multiple of $2$ or $3$, for $q$ big enough depending on $N$.

\begin{thm} \label{teo 2u3v}
Let $u\ge 2$ and $v\ge 5$. Then,

(a) there are sequences $\mathcal{F}_u$ of asymptotically good $2u$-block transitive codes over $\ff_q$ for every $q$ odd with 
$q \ge  4(u+2)^2$ and,

(b) there are sequences $\mathcal{F}_v$ of asymptotically good $3v$-block transitive codes over $\ff_q$ for every  $q=2^{2s}$ even with $q \ge (4v+9)^2$  if $v\equiv 0,1$ mod $3$ and $q \ge 4(2v+4)^2$ if $v\equiv 2$ mod $3$. 
\end{thm}

\begin{proof}
(a) Let $q$ be odd. Given $u$, equation \eqref{sigmao1} holds for at least one $m$ if $3u-2\sqrt{2u}>1$, hence $u\ge 2$. 
So, assuming $u\ge 2$, we have 
\begin{equation} \label{equ 1} 
\lceil 2\sqrt{2u} + u+2\rceil \le m -\varepsilon_2(m) \le 4u+1.
\end{equation}
Note that we can always take $m=2u+5$ if $m$ is odd or $m=4u+2$ if $m$ is even.  

Now, consider the hyperelliptic curve $C$ given by the equation 
$$y^2= h(x)= (x-\alpha_1) \cdots (x-\alpha_m)$$ 
where $\alpha_i \in \ff_q[x]$. 
It is well known that $C$ has genus $g=[\tfrac{m-1}2]$.
Without lost of generality, we will assume that $m$ is odd. 

The $\ff_q$-rational points of $C$ are 
$C(\ff_q)=\{(\alpha,\beta) \in \ff_q^2: h(\alpha)=\beta^2 \}$ 
and by the Hasse-Weil bound the number $N$ of them is bounded by
$$N \ge q+1-2g\sqrt q= q+1-(m-1)\sqrt q.$$
There are exactly $m$ rational points of the form $(\alpha,0)$, i.e.\@ the roots of $h(x)$, plus the point at infinity $\infty$. 
To apply Theorem \ref{goodAGanychar} we need a set of cardinality $2u$ inside $S_q^2(h)$, so we look for rational points of the form $(\alpha,\beta)$ with $\beta\ne 0$ such that $h(\alpha)=\beta^2 \ne 0$. This is equivalent to solve the inequality
$$\{q+1-(m-1)\sqrt q\} -(m+1) \ge  2u.$$ 
Taking $x^2=q$, we have to solve for the quadratic inequality 
\begin{equation} \label{quadratic}
x^2-(m-1)x-(m-2u) \ge 0.
\end{equation}
Taking $m=2u+5$, \eqref{quadratic} holds for $x\ge m=2(u+2)$, that is for every $q> 4(u+2)^2$ odd.

\sk 
(b) For $q=2^{2s}$ we proceed similarly as before. Starting from \eqref{sigmaed=3}, since $v\ge 5$, we can take $m=2v+5$ odd. 
Consider now the superelliptic curve $C$: $y^3=h(x) =  (x-\alpha_1) \cdots (x-\alpha_m)$. By the Riemann-Hurwitz formula, the genus of $C$ is given by
$$g=\tfrac 12 \{3(|R|-2)-|R|\}+1=|R|-2$$
where $R=\{a_1,\ldots,a_m\}$ if $3\mid m$ and $R=\{a_1,\ldots,a_m\}\cup\{\infty\}$ if $3\nmid m$. 
So, if $v \equiv 0,1$ mod $3$, then $3 \mid m$ and $g=m-1$ while if $v\equiv 2$ mod $3$ then $3\nmid m$ and $g=m-1-\epsilon_3(m)=m-2$.
The number of rational points of $C$ is $N\ge q+1-2(m-1)\sqrt q$. Taking $q=x^2$, we have to solve the inequality
$$x^2+1-2(m-1-\varepsilon_3(m))x-(m+1)\ge 3v.$$ 
Thus, if $v\equiv 0,1$ mod $3$ then $\epsilon_3(m)=0$ and we have to solve $x^2-2(m-1)x-\tfrac 52 (m-3)\geq 0$. 
Since $m\geq 15$ we can take $x \ge 2m-1=4v+9$ and so we are led to consider the fields $\ff_q$ with $q\ge (4v+9)^2$  of the form $2^{2s}$.
Similarly, if $v\equiv 2$ mod $3$ we get $x\ge 2(m-1)=2(2v+4)$ and the result follows.
\end{proof}

\begin{exam} \label{ej grbtc}
For instance, if $q$ is odd there are families of asymptotically good $4$-block transitive codes over $\ff_{3^5}$. 
We give a table with the first values for  $u$ and $q$ such that there are good families of $2u$-block transitive codes over $\ff_q$, for the first prime characteristics, we give also the first prime field for whichthe above construction works.
$$\begin{tabular}{|c|c|c|c|}
\hline 
$u$ & $2u$-bt codes & $4(u+2)^2$ & fields \\
\hline
$2$ & $4$-bt codes & $64$  & $\ff_{3^4}, \ff_{5^3}, \ff_{7^3}, \ff_{11^2}, \ff_{13^2}, \ff_{67}$ \\
$3$ & $6$-bt codes & $100$ & $\ff_{3^5}, \ff_{5^3}, \ff_{7^3}, \ff_{11^2}, \ff_{13^2}, \ff_{101}$ \\
$4$ & $8$-bt codes & $144$ & $\ff_{3^5}, \ff_{5^4}, \ff_{7^3}, \ff_{11^2}, \ff_{13^2}, \ff_{149}$ \\
$\cdots$ &$\cdots$ & $\cdots$ & $\cdots$ \\
$12$ & $24$-bt codes & $784$ & $\ff_{3^7}, \ff_{5^5}, \ff_{7^4}, \ff_{11^3}, \ff_{13^3}, \ff_{787}$ \\
\hline
\end{tabular}$$
For $q$ even and $v\in\{3,4,5\}$ there are families of good  $9$ and $15$-block transitive codes over $\ff_{2^{10}}$ and $12$-block transitive codes over $\ff_{2^{12}}$. In general by taking, for example, $v=2^{2r}-1$ for $r\geq 1$ we have that $v\equiv 1\mod 3$ and $(4v+9)^2=(2^{2r+2}+5)^2$. Since $2^{2r+3}>2^{2r+2}+5$ we see that  there are good families of $3v$-block transitive codes over $\ff_{2^{2(2r+3)}}$. 
\end{exam}

\section{Good block transitive codes over prime fields}
Here we focus in more detail on the construction of asymptotically good sequences of block transitive codes over prime fields $\ff_p$. 
Theorem \ref{teo 2u3v} assures that good sequences of both $2u$-block transitive codes over $\ff_p$ exist, but the minimum possible value of $p$
grows quadratically with $u$. We want to find good sequences of $r$-block transitive codes for small values of $r$ over the smallest prime fields possible. In particular, we want to produce good $4$-block transitive codes (which are the more cyclic-like we can get) 
over small prime fields.

\subsubsection*{Restriction of good codes to prime fields}
Let $p$ be a prime. We respectively denote by $k'$ and $d'$ the dimension and minimum distance of the restriction code $\cod'$ over 
$\ff_p$ of an $[n,k,d]$-code $\cod$ over $\ff_{p^s}$. We clearly have $d' \geq d$ and, from \cite[Lemma 9.1.3]{Sti}, also 
\begin{equation} \label{k'} 
k'\geq k-(s-1)(n-k)=sk-(s-1)n.
\end{equation}
Consider a sequence of $[n_i,k_i,d_i]$-codes $\{\cod_i\}$ over $\ff_{p^s}$ attaining an $(\ell, \delta)$-bound. Then, we have that
$\frac{k_i}{n_i} \geq \ell-\delta$ and $\frac{d_i}{n_i}\geq \delta$, for $i$ big enough. Thus, by \eqref{k'}, 
\begin{equation} \label{k''} 
R_i' = \frac{k_i'}{n_i}\geq\frac{sk_i}{n_i}-s+1 \geq s\ell-s+1-s\delta \qquad \text{and} \qquad 
\delta_i' = \frac{d_i'}{n_i}\geq \frac{d_i}{n_i} \ge \delta.
\end{equation}
If $\{\cod_i\}_{i=1}^\infty$ also attains the TVZ-bound over $\ff_{p^{2m}}$, we have 
$$ \alpha_{p^{2m}}(\delta)\geq 1-\frac{1}{p^m-1}-\delta$$
where $\ell=1-\frac{1}{p^m-1}$.
Then, by \eqref{alfaq}, \eqref{Aq2} and \eqref{k''}, the following lower bound for Manin's function over the prime field 
$\ff_p$ holds  
\begin{equation} \label{cota alfa}
 \alpha_p(\delta) \geq  1-\frac{2m}{p^m-1} - 2m \delta.
\end{equation}
That is, we cannot assure that the restricted sequence $\{\cod_i'\}$ of codes satisfy an $(\ell',\delta')$-bound, but they still have a good performance. Furthermore, we have the following.

\begin{lem} \label{lemin tvz}
Let $p$ be a prime and $\mathcal{G}=\{\cod_i\}_{i=1}^\infty$ a sequence of codes over $\ff_{p^{2m}}$ attaining the TVZ-bound. 
Then, if $p\ge 5$, the sequence of restricted codes $\mathcal{G}'=\{\cod_i'\}_{i=1}^\infty$ is asymptotically good over $\ff_p$.
Also, $\mathcal{G}'$ is asymptotically good over $\ff_2$ if $m\ge 3$ and over $\ff_3$ if $m\ge 2$. 
\end{lem}

\begin{proof}
By \eqref{k''} we only need to show that $0 < 2m(\ell-\delta-1)+1 <1$. Since $\mathcal{G}$ attains the TVZ-bound over $\ff_{p^{2m}}$, we have $\ell=1-\tfrac{1}{p^m-1}$. This is equivalent to 
$$0< \tfrac{1}{2m}-\tfrac{1}{p^m-1}.$$
This holds for $p=2$ and $m\ge 3$, $p=3$ and $m\ge 2$ and for $p\ge 5$ for every $m$.
\end{proof}

Now, we show that there are sequences of asymptotically good $r$-block transitive codes over $\ff_p$, for every prime $p$ and some 
$r$ depending on $p$. 

\begin{prop}
Let $p$ be a prime number and $m$ an integer.
\begin{enumerate}[(a)] 
\item There are asymptotically good $2(p-1)$-block transitive codes over $\ff_p$, for any $p\ge 13$. 

\item There are asymptotically good $p^m(p^m-1)$-block transitive codes over $\ff_p$, for every $m\ge 3$ if $p=2$, $m\ge 2$ if $p=3$ and 
$m\ge 1$ if $p\ge 5$.
\end{enumerate}
\end{prop}

\begin{proof}
The result follows directly from Theorems \ref{goodag-psquare} and \ref{goodag psquare2} and Lemma \ref{lemin tvz}.
\end{proof}

\begin{rem}
Notice that by the results in \cite{Sti06} and using Lemma \ref{lemin tvz} we can assure the existence of asymptotically good sequences of transitive codes and self-dual codes over prime fields. 
\end{rem}

\subsubsection*{Good $4$-block transitive codes} 
Now we present a direct construction of good block-transitive AG-codes over certain prime fields by using Theorem \ref{goodAGanychar} 
and properties of the Legendre symbol. We focus on the case of $4$-block transitive codes.
We have seen in the previous section (Example \ref{ej grbtc}) that good $4$-block transitive codes exist for every prime field $\ff_p$ with $p\ge 67$. Here we will prove that they also exist for every prime $p\ge 13$ by using properties of Legendre symbols.

Let $m=2n+1$ and let us consider the set
\[R_n=\{\alpha_1,\ldots,\alpha_n\}\subset\ff_q, \]
such that $\alpha_i^{-1}\notin R_n$ for $i=1,\ldots,n$, where $q$ is odd. We define the separable polynomial of degree $m$
\begin{equation} \label{poly m} 
h(t)=(t+1)\prod_{i=1}^n(t-\alpha_i)(t-\alpha_i^{-1}).
\end{equation}
Since we are interested in $4$-block transitive codes we need,  according to Theorem 61, a separable polynomial over $\ff_q$ of degree $m$ that splits into linear factors over $\ff_q$ and a set $\Sigma_{o}\subset S^2_q$ with two elements. Notice that if we take a polynomial $h(t)$ as in \eqref{poly m} then $h(t)$ satisfies the requirements and since $h(0)=1$ we also have that $0\in \Sigma_o$. Thus if we choose a polynomial as in \eqref{poly m} we just need one more element in $\Sigma_o$. On the other hand, from condition (30)  of Theorem 61 for $u=2$ we see that we have to  take $m=9$.

%
%
%
%

\begin{prop}
There are asymptotically good $4$-block transitive codes over $\ff_p$, for every prime $p\ge 13$. 
\end{prop}

\begin{proof}
By the above comments we consider a polynomial $h(t)$ as in \eqref{poly m} of degree $9$ so that $n=4$. It remains to find an element $0\neq\beta \in \ff_p$ such that 
$h(\beta)$ is a nonzero square in $\ff_p$.

It is easy to check that for $p \le 11$ there is no separable polynomial of degree $9$ satisfying the required conditions. 
Let $p=13$, $17$, $19$ and $23$. It is straightforward to check that the elements $\alpha_1=2$, $\alpha_2=3$, $\alpha_3=4$ and 
$\alpha_4=5$ of $\ff_p$ with $p=13, 17$ or $23$ satisfy $\alpha_i^{-1} \notin R_4=\{2,3,4,5\}$ for $i=1,\ldots,4$. 
and
	\begin{align*} 
		& h(t) = (t+1)(t-2)(t-7)(t-3)(t-9)(t-4)(t-10)(t-5)(t-8) \in \ff_{13}[t],  \msk \\ 
          		& h(t)=(t+1)(t-2)(t-9)(t-3)(t-6)(t-4)(t-13)(t-5)(t-9) \in \ff_{17}[t], \msk \\ 
		& h(t) = (t+1)(t-2)(t-12)(t-3)(t-8)(t-4)(t-6)(t-5)(t-14) \in \ff_{23}[t].  
	\end{align*}
	Thus, by taking $\alpha=11\in \ff_{13}$ we have that $h(11)=3=4^2$ in $\ff_{13}$. Similarly, by choosing $\alpha=1 \in \ff_{17}$ and 
	$\alpha=7 \in \ff_{23}$ we get $h(1)=13=8^2$ in $\ff_{17}$ and $h(7)=3=7^2$ in $\ff_{23}$, respectively.
	
On the other hand, the inverses of the elements $\alpha_1=2$, $\alpha_2=3$, $\alpha_3=4$ and $\alpha_4=6$ of $\ff_{19}$ are different from $2,3,4,6$ and 
	\[ h(t) = (t+1)(t-2)(t-10)(t-3)(t-13)(t-4)(t-5)(t-6)(t-16)\,,\]
	so that by taking $\alpha=12\in \ff_{19}$ we have that $h(12)=4=2^2$ in $\ff_{19}$.

Therefore, there are sequences of $4$-block transitive codes over $\ff_p$ for $p=13, 17, 19$ and $23$ which are asymptotically good. 
Proceeding similarly one can find sets $R_4 \in \ff_p$ an a nonzero element $\beta \in \ff_p$ such that $h(\beta)$ is a nonzero square for $p$ prime with $p\le 61$. 
The result holds for $\ff_p$ with $p\ge 67$ by Theorem \ref{goodAGanychar} and Example \ref{ej grbtc}, and thus we are done. 
\end{proof}

For small values of $r$, we will give asymptotically good sequences of $r$-bt codes over $\ff_p$, generated by explicit sets $R_n$ and $S_u$.

We now consider the case of an arbitrary prime $p\ge 29$. In this case we can take the same set of roots $R_4=\{2,3,4,5\}$ of $h$ for every $\ff_{p^s}$, 
$s$ odd. Also, by Fermat's theorem the inverse of $a$ in $\ff_p$ is $a^{p-2}$ and, hence, 
	\eqref{poly m} takes the form
\begin{equation} \label{poli k} 
h(t) = (t+1) \prod_{k=2}^5 (t-k)(t-k^{p-2})
		\in \ff_{p^s}[t]
\end{equation}
with $s$ odd.	We want to find an element $a\in \ff_{p^s}^*$, $s$ odd, such that $h(a)$ is a non zero square in $\ff_p$.
	It suffices to check that $\big( \frac{h(a)}{p} \big) = 1$, where $(\frac{\cdot}{p})$ denotes the Legendre symbol modulo $p$ (Jacobi symbol for $p$ odd not prime). Since the Jacobi and Legendre symbols are multiplicative, we have
\begin{equation} \label{poli k leg} 
\legendre{h(t)}{p^s} = \legendre{t+1}{p} \prod_{k=2}^5  \legendre{t-k}{p} \legendre{t-k^{p-2}}{p}
\end{equation} 
	since $s$ is odd and $(\tfrac{\cdot}{p})^2=1$.	
	
	Evaluating $h(t)$ at $t=p-j$ for $2 \le j \le \lfloor \tfrac{p-1}5 \rfloor$ we ensure that $h(p-j)\ne 0$. Thus, 
\begin{equation} \label{h(p-j)}
h(p-j) = 	(p-(j-1)) \prod_{k=2}^5 (p-(j+k))(p-(j+k^{p-2})) \ne 0.
\end{equation}
	Now, computing the Jacobi symbol of $h(p-j)$, we get 
		$$\legendre{h(p-j)}{p^s} = \legendre{1-j}{p} \prod_{k=2}^5  \legendre{j+k}{p} \legendre{k}{p}^2 \legendre{j+k^{p-2}}{p}  
	=  \legendre{1-j}{p} \prod_{k=2}^5  \legendre{j+k}{p} \legendre{k}{p} \legendre{kj+1}{p}.$$
Note that this last expression involves only $k$ and not its inverse.
	For instance, for $p\ge 37$ we can take the $j=2,\ldots, 7$ and 
     \begin{equation} \begin{split} \label{bunch of legendres}
		 & \legendre{h(p-2)}{p^s} = \legendre{-1}{p} \legendre{5}{p} \legendre{11}{p}, \\  
		 & \legendre{h(p-3)}{p^s} = \legendre{-1}{p} \legendre{2}{p} \legendre{5}{p} \legendre{13}{p}, \\
		 & \legendre{h(p-4)}{p^s} = \legendre{-1}{p} \legendre{2}{p} \legendre{5}{p} \legendre{13}{p} \legendre{17}{p}, \\
		 & \legendre{h(p-5)}{p^s} = \legendre{-1}{p} \legendre{11}{p}\legendre{13}{p},  \\ 
		 & \legendre{h(p-6)}{p^s} = \legendre{-1}{p}\legendre{2}{p}\legendre{3}{p} \legendre{5}{p}\legendre{11}{p} \legendre{13}{p}\legendre{19}{p}\legendre{31}{p}, \\
		 & \legendre{h(p-7)}{p^s} = \legendre{-1}{p} \legendre{2}{p} \legendre{29}{p}, 
     \end{split}
     \end{equation}
     for every $s\ge 1$ odd.	
	This reduces the search of a non zero element $\alpha \in \ff_p$ such that $h(\alpha)$ is a non zero square in $\ff_p$ to the computation of Legendre symbols
	for a given prime $p\geq 37$.

\begin{coro} \label{coro ps}
There are asymptotically good 4-block transitive AG-codes over $\ff_{p^s}$, for any $s\ge 1$ odd, generated by the sets $R_4=\{2,3,4,5\}$ and $\Sigma_o=\{0,p-j\}$ in $\ff_p$, with $2 \le j \le 7$, for infinitely many prime numbers $p$. 
For instance, for every prime of the form $p=220k+\ell$, with $k\in\na$ and $\ell \in \{1,9,11,19\}$, the element $p-2$ together with the set $S$ generates an asymptotically good sequence of $4$-block transitive codes over $\ff_{p^s}$ for every $s$ odd. 
\end{coro}
\begin{proof}
	Consider the polynomial $h(t)$ given in \eqref{poli k}. 
    It suffices to find infinitely many primes $p$, such that $\big( \tfrac{h(p-j)}{p} \big) = 1$, for a given $j$.
	Consider first the case $j=2$. We look for prime numbers $p$  such that  
	$$ \legendre{h(p-2)}{p} = \legendre{-1}{p} \legendre{5}{p} \legendre{11}{p} = 1. $$
	We have $\legendre{-1}{p}=(-1)^{\frac{p-1}2}$ and $\legendre{5}{p}=(-1)^{\lfloor \frac{p-1}2 \rfloor}$.
     Hence, $\legendre{-1}{p}=1$ if $p\equiv 1$ mod $4$, $\legendre{-1}{p}=-1$ if $p\equiv 3$ mod $4$, $\legendre{5}{p}=1$ if $p\equiv \pm 1$ mod $5$ and 
$\legendre{5}{p}=-1$ if $p\equiv \pm 2$ mod $5$.     
As a first step, we find $p$'s satisfying $\legendre{-1}{p} \legendre{5}{p}=1$. The values $\legendre{-1}{p} =\legendre{5}{p}=1$ are obtained by taking 
$p=20k+1$ or $p=20k+9$ with $k\in \na$, while $\legendre{-1}{p} =\legendre{5}{p}=-1$ are obtained with $p=20k+3$ or $p=20k+7$ with $k\in \na$. 
Now, take $p=(20\cdot 11) k+j$ with $j=1,3,7$ or $9$ and $k\in \na$. 
By the quadratic reciprocity law, we have 
\begin{equation} \label{QR}
\legendre{11}{p} = (-1)^{\frac{(11-1)}2 \frac{(p-1)}2} \legendre{(20\cdot 11)k+1}{11} =\legendre{1}{11}=1
\end{equation}
for $j=1,9$. 
Similarly, we can do the same considering the cases of different parity in the symbols $\legendre{-1}{p}$ and $\legendre{5}{p}$.
The pattern $\legendre{-1}{p}=-1$ and $\legendre{5}{p}=1$ holds for $p=20k+11$ or $20k+19$, $k\in \na$. 
By quadratic reciprocity, \eqref{QR} above is also true for these values of 
$p$ (however, the pattern $\legendre{-1}{p}=1$, $\legendre{5}{p}=-1$ does not lead to a solution of \eqref{QR}). 

That is, for every prime number of the form $p=220k+j$, with $j=1,9,11,19$ and $k\in \na$, we have that 
	$\legendre{h(p-2)}{p}=1$. 
    Proceeding similarly as before, from the expressions in \eqref{bunch of legendres} one can prove similar results for $R=\{p-j\}$ with $3\le j\le 7$ and we leave it to the reader. 
By Dirichlet's theorem on arithmetic progressions, there are infinitely many prime numbers $p$ of the form $220k+1$, $220k+9$, $220k+11$ and $p=220k+19$ with $k\in \na$, and thus the result follows.
\end{proof}

Notice that by using the same method as above, one can obtain asymptotically good sequences of $6$-block transitive codes over 
$\ff_{p^s}$ for every $s\ge 1$ odd, by considering any two of the expressions in \eqref{bunch of legendres} simultaneously. But now, since $u=3$, we have to take $m=11$ 
(by \eqref{sigmao1}) and hence we have to consider the polynomial $h(t)$ in \eqref{poli k} with $n=6$, that is \eqref{poli k leg} with $2\le k \le 6$. In other words,
$$\legendre{h_6(p-j)}{p^s} = \legendre{1-j}{p} \prod_{k=2}^6  \legendre{j+k}{p} \legendre{k}{p} \legendre{kj+1}{p} = 
\legendre{1-j}{p} \legendre{h_5(p-j)}{p^s} \legendre{j+6}{p} \legendre{6}{p} \legendre{6j+1}{p}. $$
So, for instance for $j=2,3$, by \eqref{h(p-j)}, we have
\begin{align*}
& \legendre{h_6(p-2)}{p^s} = \legendre{-1}{p} \legendre{5}{p} \legendre{11}{p} \legendre{8}{p} \legendre{6}{p} \legendre{13}{p} = \legendre{-1}{p} \legendre{3}{p} \legendre{5}{p} 
   \legendre{11}{p} \legendre{13}{p} , \\  
& \legendre{h(p-3)}{p^s} =  \legendre{-1}{p} \legendre{2}{p} \legendre{5}{p} \legendre{13}{p} \legendre{9}{p} \legendre{6}{p} \legendre{19}{p} = \legendre{-1}{p} \legendre{3}{p} \legendre{5}{p} 
\legendre{13}{p} \legendre{19}{p}.
\end{align*}
It is known that $\legendre{3}{p}=(-1)^{\lfloor \tfrac{p+1}6\rfloor}$, i.e.\@ $\legendre{3}{p}=1$ if $p\equiv \pm 1$ mod $12$ and $p\equiv \pm 5$ mod $12$.
Hence, for a prime of the form $p=60k+1$ we have that  $\legendre{-1}{p} \legendre{3}{p} \legendre{5}{p}=1$. So, consider a prime of the form 
$$p=(11\cdot 13\cdot 19 \cdot 60) k +1 = 163021k+1, \qquad k\in \na.$$ 
By applying quadratic reciprocity again, we get $\legendre{11}{p} =  \legendre{13}{p} =  \legendre{19}{p} =1$. 
In this way, for primes of the form $p=(11\cdot 13\cdot  19 \cdot 60) k +1$, $k\in\na$, the sets $S=\{2,3,4,5,6\}$ and $R=\{0,p-2,p-3\}$ generates an asymptotically good sequence of $6$-block transitive codes over $\ff_{p^s}$ for every $s$ odd. In particular, for $k=1$ we get a good family of $6$-block transitive codes over $\ff_{163021}$.

\bibliographystyle{plain}

\def\cprime{$'$}

\end{document}